\documentclass[11pt]{article}%
\usepackage{graphicx}
\usepackage{subcaption}
\usepackage{makeidx}
\usepackage{amssymb}
\usepackage[all]{xy}
\usepackage{float}
\usepackage[latin1]{inputenc}
\usepackage{amsfonts}
\usepackage{amsmath}
\usepackage{amssymb}
\usepackage{url}
\usepackage{hyperref}
\usepackage{mathabx}
\usepackage{wrapfig}
\usepackage{dsfont} 
\usepackage{resizegather}
\usepackage{yfonts}

\usepackage{amsthm}
\usepackage{url}
\usepackage{hyperref}
\providecommand{\U}[1]{\protect\rule{.1in}{.1in}}

\usepackage{xcolor}

\newcommand{\hooklongrightarrow}{\lhook\joinrel\longrightarrow}

\newtheorem{prop}{Proposition}[section]
\newtheorem{cor}[prop]{Corollary}

\newtheorem{lem}[prop]{Lemma}

\newtheorem{theo}[prop]{Theorem}

\newcommand{\un}{\mathds{1}}

\newcommand{\EE}{\mathbb{E}}

\newcommand{\HH}{\mathbb{H}}

\newcommand{\LL}{\mathbb{L}}

\newcommand{\PP}{\mathbb{P}}
\newcommand{\QQ}{\mathbb{Q}}
\newcommand{\RR}{\mathbb{R}}

\newcommand{\UU}{\mathbb{U}}
\newcommand{\VV}{\mathbb{V}}
\newcommand{\XX}{\mathbb{X}}

\newcommand{\ZZ}{\mathbb{Z}}
\newcommand{\GG}{\mathbb{G}}

\newcommand{\Na}{ {\cal N }}
\newcommand{\Ka}{ {\cal K }}

\newcommand{\Ea}{ {\cal E }}

\newcommand{\Ra}{ {\cal R }}
\newcommand{\Va}{ {\cal V }}
\newcommand{\Ua}{ {\cal U }}
\newcommand{\Fa}{ {\cal F }}
\newcommand{\Ga}{ {\cal G }}

\newcommand{\Ma}{ {\cal M }}

\newcommand{\Pa}{ {\cal P }}
\newcommand{\Za}{ {\cal Z }}
\newcommand{\Ya}{ {\cal Y }}
\newcommand{\Wa}{ {\cal W }}

\newcommand{\point}{\mbox{\LARGE .}}
\newcommand{\cqfd}{\hfill\blbx \\}
\def\blbx{\hbox{\vrule height 5pt width 5pt depth 0pt}\medskip}

\def \PP{\mathbb{P}}
\def \RR{\mathbb{R}}

\def \EE{\mathbb{E}}

\def \LL{\mathbb{L}}
\def \ZZ{\mathbb{Z}}

\begin{document}

  \title{A theoretical analysis of one-dimensional discrete generation ensemble Kalman particle filters}
  \author{P. Del Moral, E. Horton}


\maketitle

\begin{abstract}

Despite the widespread usage of discrete generation Ensemble Kalman particle filtering metho\-dology to solve nonlinear and high dimensional filtering and inverse problems, little is known about their mathematical foundations. As genetic-type particle filters (a.k.a. sequential Monte Carlo), this ensemble-type  methodology  can also be interpreted as mean-field particle approximations of the Kalman-Bucy filtering equation.
In contrast with conventional mean-field type interacting particle methods equipped with a globally Lipschitz interacting drift-type function, Ensemble Kalman filters depend on a nonlinear and quadratic-type interaction function defined in terms of the  sample covariance of the particles. 

Most of the  literature in applied mathematics and computer science on these sophisticated interacting particle methods amounts to designing different classes of useable observer-type particle methods. These methods are based on a variety of inconsistent but judicious ensemble auxiliary transformations or include additional inflation/localisation-type algorithmic innovations, in order to avoid the inherent time-degeneracy of an insufficient particle ensemble size when solving a filtering problem with an unstable signal. 

To the best of our knowledge, the first and the only rigorous mathematical analysis of these sophisticated discrete generation particle filters is developed in the pioneering  articles by Le Gland-Monbet-Tran and by Mandel-Cobb-Beezley, which were published in the early 2010s. 
Nevertheless, besides the fact that these studies prove the asymptotic consistency of the Ensemble Kalman filter, they provide exceedingly pessimistic mean-error estimates that grow exponentially fast with respect to the time horizon, even for linear Gaussian filtering problems with stable one dimensional signals.

 In the present article we 
develop a novel self-contained and complete stochastic perturbation analysis of the fluctuations, the stability, and the long-time performance of these discrete generation  ensemble Kalman particle filters, including  time-uniform and non-asymptotic mean-error  estimates that apply to possibly unstable signals. To the best of our knowledge, these are the first results of this type in the literature on discrete generation particle filters, including the class of genetic-type particle filters and discrete generation ensemble Kalman filters. The stochastic Riccati difference equations considered in this work are also of interest in their own right, as a prototype of a new class of stochastic rational difference equation. 

\newline


\end{abstract}

\newpage
\section{Introduction}
The Ensemble Kalman filter ({abbreviated EnKF}) is a class of interacting particle system methodologies for solving nonlinear filtering and inverse problems. They were introduced by Evensen in the seminal article~\cite{evensen-intro} published in 1994, see also~\cite{evensen-review,evensen-book} for a more recent overview.  In the last three decades the EnKF methodology has become one of the most used numerical techniques
for solving high dimensional forecasting and data assimilation problems in a variety of applications including ocean and atmosphere
sciences~\cite{allen,lisa,majda,kalnay,ott}, fluid mechanics~\cite{beyou,memin-1,memin-2}, image inverse problems~\cite{beyou-2}, weather forecasting~\cite{anderson-jl,anderson-jl-2,burgers,houte},  environmental and ecological statistics~\cite{eknes,johns} and oil reservoir simulations~\cite{evensen-reservoir,nydal,seiler,skj,weng}. 
To connect these EnKF techniques with particle filtering methodology for solving high dimensional problems arising in fluid mechanics, we also refer to~\cite{kantas}. 

In contrast with genetic-type particle filters (a.k.a. Sequential Monte Carlo, often abbreviated SMC), the EnKF is defined by a system of particles evolving as the signal in some state space with an interaction function that depends on the sample covariance matrices of the system. 
 For a detailed  mathematical description of particle filters and ensemble Kalman filter methodologies in both continuous and discrete time settings we refer, for instance, to the book~\cite{dm-penev} and the references therein. See also  section~\ref{sec-EnKF-intro} in the present article dedicated to nonlinear Kalman-Bucy type Markov chains and their mean-field particle interpretations.

 There exists a vast literature in data assimilation theory dedicated to the
numerical analysis of the EnKF view as a useable observer-type algorithm with a very small sample size and equipped with a variety of judicious ensemble square-root type transformations, as well as inflation/localisation-type algorithmic innovations to control the unstable directions of the signal, see for instance~\cite{anderson-3,anderson-4,kelly-law-stuart,kelly-law-tong,tong-16,tong}, as well as~\cite{bishop-18,bishop-20} and the references therein. From this point of view, the EnKF is no longer interpreted as a true approximation of the optimal filter but as an observer. Here, the terminology ``well-posedness" is not understood as a traditional mathematical consistency-type  property of some Monte Carlo statistical estimate but as a lack of divergence of the algorithm with respect to the time horizon.

Despite the widespread use of EnKF techniques to solve nonlinear and high dimensional filtering and inverse problems, little is known about their theoretical performance and their mathematical foundations. 
The first rigorous mathematical analysis of discrete generation EnKF appeared in 2011 in the independent pioneering works of 
 Le Gland,  Monbet and Tran~\cite{legland}, and Mandel,  Cobb and Beezley~\cite{mandel}.
These two seminal articles provide mean error estimates for discrete generation EnKF and show that the EnKF converges towards the Kalman filter for linear-Gaussian problems as the number of samples, say $N$, tends to infinity. The article~\cite{legland} also shows that the EnKF doesn't converge to the optimal filter for nonlinear or non-Gaussian filtering problems. 

 The consistent-type EnKF methodology for linear-Gaussian models can be extended to 
 nonlinear filtering problems  using the
 new class feedback particle filter methodology introduced by Mehta and Meyn and their co-authors in the series of seminal articles~\cite{prashant-1,prashant-2,prashant-3,prashant-4,prashant-5}. To obtain a consistent EnKF for nonlinear filtering problems one needs to solve a Poisson-type equation at every time step, which often requires an additional level of approximation. 
 The theoretical analysis of this type of sophisticated nonlinear EnKF is not discussed in the present article but we refer the reader to~\cite{prashant-5,prashant-6,tag-1} for the analysis of its performance and the convergence.

Besides these fundamental theoretical advances in the understanding of EnKF filters in linear-Gaussian models, the non-asymptotic analysis developed in the aforementioned articles yields exceedingly pessimistic estimates that grow exponentially fast with respect to the time horizon. Note that for a  time horizon $t=39$ an innocent exponential bound of the form $55\times e^{5 t}/N\geq 10^{86}/N$ will require a sample size $N$ that is $10$ times larger than the number $10^{86}$ of elementary particles of matter in the visible universe in order to obtain a poor $10\%$ accuracy after $39$ runs. 

The mathematical analysis of the continuous-time version of the EnKF methodology has started more recently in~\cite{DelMoral/Tugaut:2016,dm-k-tu}, followed by the series of articles~\cite{BishopDelMoralMatricRicc,dm-b-niclas,bishop-19}, see also review article~\cite{bishop-20}. Some extensions to nonlinear filtering problems are also developed in the series of more recent articles~\cite{lange-1,lange-2,lange-3,lange-4,sahani}. In contrast with the 
feedback particle filter methodology discussed above, these articles are concerned with a class of  EnKF particle filters that are only consistent for linear-Gaussian models. As expected, for nonlinear or non-Gaussian filtering problems, none of these EnKF converge to the optimal filter as the number of particles tends to infinity.

From a probabilistic viewpoint, in the linear-Gaussian case, continuous time  EnKF  are represented by a Kalman-Bucy-type diffusion coupled with a stochastic Riccati matrix nonlinear diffusion equation (see for instance Theorem 3.1~in~\cite{DelMoral/Tugaut:2016}).
In the series of articles~\cite{DelMoral/Tugaut:2016,dm-k-tu,BishopDelMoralMatricRicc,dm-b-niclas,bishop-19},  the authors present a refined stochastic stability analysis of these rather sophisticated diffusion-type perturbation models, including central limit theorems, as well as uniform bias and mean error estimates with respect to the time horizon for several classes of consistent EnKF stochastic models. 
 
To describe briefly these results, as underlined in the review article~\cite{bishop-20}, we emphasise that the continuous-time EnKF methodology (for linear-Gaussian models) may be broadly divided into three different classes of probabilistic models, according to the level of fluctuation added via sampling noise needed to ensure that the EnKF sample mean and covariance are consistent in the linear-Gaussian setting.
By chronological order of appearance in the literature, these three different classes of continuous-time EnKF are briefly discussed below.

The continuous time version of the EnKF discussed in section~\ref{sec-EnKF-intro} in the present article coincides with the so-called ``Vanilla'' discrete time original form of the EnKF~\cite{evensen-intro,burgers}. This first class of 
continuous-time EnKF exhibits the most fluctuations due to sampling both signal and observation noises (see for instance the evolution equation (\ref{kalman-EnKF-def})).

The second class of EnKF is a continuous-time version of the Sakov and Oke square root  filter (a.k.a. ``deterministic EnKF'') introduced in~\cite{sakov,sakov-2}, see also~\cite{reich,reich-2}. 
Unfortunately, this class of 
``deterministic EnKF'' is only consistent  for continuous time models. In the discrete time situation, it fails to converge to the optimal filter
as the number of particles tends to infinity, even in linear-Gaussian settings, see for instance~\cite{bishop-20} and references therein. This class of discrete generation EnKF is not discussed in the present article.

 Finally, the third class consists of purely deterministic transport-inspired EnKF with randomness coming only from the initial conditions. By construction, this class of filters can be analysed directly using the evolution semigroup of the optimal filter and that of the Riccati matrix equations. Further details on this subject can be found in the review article~\cite{bishop-20} on continuous time EnKF particle filters  and in the articles~\cite{ap-2016,floquet} dedicated to the stability of Kalman-Bucy diffusion processes and related Riccati matrix differential equations.   
 
 The stochastic analysis of the Riccati matrix diffusion associated with the sample covariance of these three classes of continuous-time EnKF is rather well understood for multivariate models. The articles~\cite{DelMoral/Tugaut:2016,dm-k-tu,BishopDelMoralMatricRicc,dm-b-niclas,bishop-19}  present several multivariate fluctuation theorems for Riccati matrix diffusions, as well as a refined non-asymptotic analysis, including several time-uniform mean error estimates.
 The extension of these uniform results to the Vanilla EnKF is developed in~\cite{DelMoral/Tugaut:2016,dm-k-tu} but only for stable and ergodic signals.
 The stability analysis and the long-time performance of Vanilla EnKF with possibly unstable and multivariate transient signals remains partially understood. 

To understand the difficulties that arise these problems it is worth mentioning that
even for elementary one-dimensional problems, the stochastic Riccati diffusions describing the random evolution of the sample variances of the EnKF may exhibit surprisingly 
heavy tailed or Gaussian tailed invariant distributions, depending on the class of EnKF one uses as well as the number of samples~\cite{bishop-19,bishop-20}.  In this setting, some raw moments of the sample variances of the Vanilla continuous time EnKF for one-dimensional filtering problems are infinite, for any chosen finite sample size lower than some critical value. 

 To the best of our knowledge, the question of finding uniform mean error estimates for the Vanilla EnKF sample means
remains
an important open research question for both discrete generation and continuous-time multivariate models. 
The only work in this direction appears to be the recent article~\cite{bishop-19}. This article  is dedicated to  the long time performance and the stability properties of the stochastic Riccati diffusion associated with the three different classes of continuous-time EnKF filters discussed above for {\em one-dimensional} and linear-Gaussian filtering problems with possibly unstable and transient signals. The extension of these results to multivariate models remains open.

The main advantage in working with continuous-time models comes from the fact that the sample mean and the sample covariance of the EnKF filters satisfy a coupled nonlinear diffusion equation which can be handled using conventional stochastic analysis. Note however that in the multivariate case, the sample covariance matrices  satisfy a stochastic Riccati equation involving matrix valued martingales which require one to develop an appropriate and more sophisticated stochastic analysis in matrix spaces~\cite{BishopDelMoralMatricRicc,dm-b-niclas}.

Physical systems and most of the filtering problems arising in signal processing and tracking are typically defined by continuous time models. Nevertheless, for obvious reasons, any numerical integration of these models including all known filtering sensors are defined by discrete-time models. As a result, the stochastic analysis used for the continuous-time EnKF is no longer applicable without adding an additional level of approximation. 

The analysis of discrete generation EnKF is more delicate. To handle the structure and the statistical difficulty of discrete time models, new stochastic tools need to be developed.
For instance, in contrast with continuous-time models, discrete generation EnKF are not defined by a single coupled diffusion process but in terms of coupled two-step prediction-updating process (a.k.a. forecasting-analysis steps in EnKF and data assimilation literature). Furthermore, the Gaussian-nature of the diffusion models arising in the analysis of continuous-time EnKF theory is also lost.  Another inherent difficulty is that discrete generation 
 EnKF  involve more sophisticated
 non-central $\chi$-square nonlinear fluctuations  (cf. for instance
Theorem~\ref{stoch-perturbation-theo} and Corollary~\ref{cor-p-chi-evol}).

We expect the multivariate version of the EnKF evolution to involve a similar two-step  prediction-updating 
transition involving sophisticated non-central Wishart matrix local fluctuations. 
Up to some additional technicalities and at the cost of multi-index and tensor theoretical notation, we believe that the analysis of multivariate and discrete generation stochastic Riccati equations can be conducted by extending the theory of discrete generation one-dimensional models developed in the present article in the spirit of the  stochastic analysis of continuous time EnKF models developed in~\cite{BishopDelMoralMatricRicc,bishop-19}. As for continuous time models, we also believe  that one cannot expect to directly deduce uniform state estimates from those at the level of the sample covariance matrices.
 
Due to these reasons, in the present article we have chosen to concentrate our study to one-dimensional models. We intend to extend these results to more sophisticated multivariate models in a subsequent article.

The main objective of this article is to develop a novel and complete stochastic analysis on the fluctuations and the long time performance of discrete generation  EnKF particle filters that applies to filtering problems involving possibly unstable signals. Our main contributions are listed succinctly below. For more precise statements of the results, we refer the reader to the series of theorems stated in section~\ref{statement-sec}.

$\bullet$ Up to some orthogonal transformations on the sample space, our first main result, Theorem~\ref{stoch-perturbation-theo}, states that the sample variances of the EnKF satisfy an autonomous stochastic Riccati-type evolution equation, which is independent of the sequence of observations, and that the EnKF sample means evolve as a Kalman filter driven by  sample variances. In addition, the local perturbations of the EnKF and the sample variance Riccati equations
 are independent. An alternative  Markov chain realisation of the stochastic Riccati difference equation is also discussed in Corollary~\ref{cor-p-chi-evol}. Theorem~\ref{stoch-perturbation-theo} is an extended version of the stochastic perturbation theorem presented in~\cite[Theorem 3.1]{DelMoral/Tugaut:2016}. The continuous time version of the sample variance equations stated in Theorem~\ref{stoch-perturbation-theo} also coincides with  the one-dimensional stochastic Riccati equation discussed in~\cite{bishop-19}, see also~\cite{BishopDelMoralMatricRicc} for the extended matrix version of these models in the context of multivariate filtering problems. 
 
 $\bullet$ Section~\ref{sec-stability-intro} is concerned with the stability properties of the optimal filter and the stochastic Riccati equations describing the evolution of the sample variances of the EnKF for possibly unstable signals. In this context, we show that
the discrete generation stochastic Riccati equations converge exponentially fast to a single invariant measure. The contraction theorem, Theorem~\ref{stab-riccati-stoch}, and the exponential semigroup estimates (\ref{expo-decay-p}) stated in Theorem~\ref{key-theo}
are partial extensions to discrete generation models of Theorem 2.2 in the article~\cite{bishop-19}, dedicated to the stability of one dimensional stochastic Riccati equations. In Theorem~\ref{stab-riccati+M-stoch} we also show that the sample variances and the difference between the sample means and the true (possibly transient) signal forms a Markov chain that converges exponentially fast to a unique invariant measure. This seems to be the first result of this kind in the field of interacting particle-type filters, including EnKF and genetic-type particle filters. 

$\bullet$ Our third main result is concerned with uniform mean-error and bias estimates with respect to the time horizon for both the sample variances and the sample means delivered by the EnKF. Theorem~\ref{theo-u-p} shows that the sample variances are uniformly close to the true optimal variances for all times, even when the signal and the ensemble are transient. The EnKF sample mean uniform accuracy for possibly unstable signals is discussed in theorem~\ref{key-theo}. This result ensures that the sample means are uniformly close to the true optimal filter at any time horizon. The continuous time version of these results are discussed in Corollary 2.4 in~\cite{bishop-19}.

$\bullet$ Last, but not least, we present a novel fluctuation analysis of EnKF filters.  
Theorems~\ref{theo-clt-p-intro} and~\ref{term-theo} present new multivariate central limit  theorems for the sample mean and the sample variance processes. These fluctuation results, at the level of the processes, are an extended multivariate version of central limit theorem presented in~\cite{bishop-19} in the context of continuous time models.

\section{Description of the models}\label{desc-sec-intro}
\subsection{The Kalman filter}
Consider a one dimensional, time homogeneous linear-Gaussian filtering model of the following form
\begin{equation}\label{lin-Gaussian-diffusion-filtering}
\left\{
\begin{array}{rcl}
Y_n&=&C\,X_n+D\,V_n\\
X_{n+1}&=&A\,X_{n}+B\,W_{n+1},
\end{array}
\right.
\qquad n \ge 0.
\end{equation}
In the above display, $(V_n,W_{n+1})$ is a sequence of $2$-dimensional 
Gaussian centered random variables with unit variance, the initial condition of the signal $X_0$ is a Gaussian random variable with mean and variance denoted by $(\widehat{X}^-_0,P_0)$ (independent of $(V_n,W_{n+1})$), and $(A,B,C,D)$ are some non-zero parameters. The latter condition can be interpreted as a controllability and observability condition for one-dimensional filtering problems.

Let $\Ya_n=\sigma\left(Y_k,~k\leq n\right)$ be the filtration generated by the observation process. The optimal filter is defined by the
conditional distribution $\widehat{\eta}_n$ of the signal state $X_n$ given $\Ya_n$. The distribution $\widehat{\eta}_n$ is a Gaussian  distribution with  conditional mean and variance  
$$
\widehat{X}_n:=\EE(X_n~|~\Ya_n)\quad\mbox{\rm and}\quad
\widehat{P}_n:=\EE\left(\left(X_n-\widehat{X}_n\right)^{2}\right).
$$ 
The optimal one-step predictor is defined by the conditional distribution $\eta_{n}$ of the signal state $X_n$ given $\Ya_{n-1}$. The distribution $\eta_n$ is also a Gaussian  distribution with conditional  mean and variance  
$$
\widehat{X}^-_n:=\EE(X_n~|~\Ya_{n-1})\quad\mbox{\rm and}\quad
P_n:=\EE\left(\left(X_{n}-\widehat{X}_{n}^{-}\right)^{2}\right).
$$ 
The updating-prediction steps  are described by the following synthetic diagram
$$
(\widehat{X}_n^-,P_n)\longrightarrow (\widehat{X}_{n},\widehat{P}_{n})\longrightarrow (\widehat{X}^-_{n+1},P_{n+1}),
$$
with the well-known updating-prediction Kalman filter equations:
\begin{equation}\label{kalman-rec}
\left\{\begin{array}{rcl}
\widehat{X}_n&=&\widehat{X}_n^-+G_n~\left(Y_n-C\widehat{X}_n^-\right) \\
 \widehat{P}_n&=&(1-G_n C)P_n
\end{array}\right.\quad\mbox{\rm and} \quad\left\{\begin{array}{rclcrcl}
\widehat{X}_{n+1}^-&=&A \widehat{X}_{n}\\
 P_{n+1}&=& A^2\widehat{P}_{n}+R,
\end{array}\right.
\end{equation}
where $R = B^2$. In the above display, $G_n$ stands for the so-called gain parameter defined by the formula
$$
G_n:=CP_n/(C^2P_n+D^2)\Longrightarrow 1-G_nC=1/(1+SP_n)\quad \mbox{\rm with}\quad S:=(C/D)^2.
$$
The evolution equations of the variance parameters $(P_n,\widehat{P}_n)$ satisfy the Riccati rational difference equations
\begin{equation}\label{def-Ricc}
P_{n+1}=\phi\left(P_n\right):=\frac{aP_n+b}{cP_n+d}\quad \mbox{\rm and}\quad
\widehat{P}_{n+1}=\widehat{\phi}\left(\widehat{P}_n\right):=\frac{\widehat{a}\,\widehat{P}_n+\widehat{b}}{\widehat{c}\,\widehat{P}_n+\widehat{d}}
\end{equation}
with the parameters
\begin{equation}\label{def-Ricc-parameters}
\begin{array}{l}
(a,b,c,d):=\left(A^2+RS,R,S,1\right)\quad \mbox{\rm and}\quad (\widehat{a},\widehat{b},\widehat{c},\widehat{d})=(A^2,R,A^2S,1+SR)\\
\\
\Longrightarrow 
ad-bc=A^2=\widehat{a}\,\widehat{d}-\widehat{b}\,\widehat{c}>0.
\end{array}\end{equation}
One-dimensional rational difference equations are rather well understood, see for instance~\cite{brand}. Section~\ref{sec-riccati} also provides a refined stability analysis and several comparison properties of rational difference equations including non-asymptotic Taylor expansions any order.

\subsection{Ensemble Kalman filters}\label{sec-EnKF-intro}
Let $(\textgoth{X}_0, \textgoth{W}_n,\textgoth{V}_n)$ be independent copies of $(X_0, W_n,V_n)$.  Consider the nonlinear Markov chain starting at $\textgoth{X}_0$ and defined sequentially for any $n\geq 0$ by the updating-correction formulae
\begin{equation}\label{kalman-X-goth}
\left\{
\begin{array}{rcl}
\widehat{\textgoth{X}}_{n}&=&\textgoth{X}_n+\textgoth{G}_{\eta_n}~(Y_n-(C\,\textgoth{X}_n+D\,\textgoth{V}_n))\quad\mbox{\rm with}\quad \textgoth{G}_{\eta_n}~:=C\textgoth{Q}_{\eta_n}/(C^2\textgoth{Q}_{\eta_n}
+D^2)\\
&&\\
\textgoth{X}_{n+1}&=& A\,\textgoth{X}_{n}+B\,\textgoth{W}_{n+1}.\\
\end{array}
\right.
\end{equation}
In the above display, $\textgoth{Q}_{\eta_n}$ stands for the variance parameter
$$
\textgoth{Q}_{\eta_n}:=\int \left(x- \int z~ \eta_n(dz)\right)^2\eta_n(dx)\quad\mbox{\rm indexed by}\quad\eta_n=\mbox{\rm Law}(\textgoth{X}_n~|~\Ya_{n-1}).
$$
Using a simple induction argument, it is straightforward to show that
$$
\eta_n=\Na\left(\widehat{X}^-_n,P_n\right)\quad \mbox{\rm and}\quad \widehat{\eta}_n=\Na\left(\widehat{X}_n,\widehat{P}_n\right)=\quad\mbox{\rm Law}(\widehat{\textgoth{X}}_n~|~\Ya_{n}).
$$
Also observe that the above stochastic model can be interpreted as a Markov chain
with a two-step updating-prediction transition  described by the following synthetic diagram
$$
\textgoth{X}_n\longrightarrow \widehat{\textgoth{X}}_{n}\longrightarrow \textgoth{X}_{n+1}.$$
Note that the updating transition $\textgoth{X}_n\longrightarrow \widehat{\textgoth{X}}_{n}$
depends on the observation $Y_n$ delivered by the sensor as well as on 
the conditional law of $\textgoth{X}_n$ given the information $\Ya_{n-1}$.

The EnKF associated to the filtering problem (\ref{lin-Gaussian-diffusion-filtering}) coincides with the mean-field particle interpretation of the nonlinear Markov chain (\ref{kalman-X-goth}).
To define these models, we let $\xi_0=(\xi_0^i)_{1\leq i\leq N+1}$ be a sequence of $(N+1)$ independent copies of  $X_0$, for some parameter $N\geq 1$. We also denote by $\Wa_n=(\Wa^i_n)_{i\geq 1}$ and $\Va_n=(\Va^i_n)_{i\geq 1}$ a sequence of independent copies of $W_n$ and $V_n$.

 The mean field particle interpretation of the nonlinear Markov chain discussed above is given by the interacting particle system defined sequentially for any $1\leq i\leq N+1$ and $n\geq 0$ by the formulae
\begin{equation}\label{kalman-EnKF-def}
\left\{
\begin{array}{rcl}
\widehat{\xi}^i_{n}&=&\xi^i_n+g_n~(Y_n-(C\xi^i_n+D\Va^i_n))\quad\mbox{\rm with}\quad g_n:=Cp_n/(C^2p_n
+D^2)\\
&&\\
\xi^i_{n+1}&=& A\,\widehat{\xi}^i_{n}+B\,\Wa^i_{n+1}.
\end{array}
\right.
\end{equation}
In the above display $p_n$ stands for the normalised sample variance
$$
p_n:=\frac{1}{N}\sum_{1\leq i\leq N+1}(\xi^i_n-m_n)^2=\left(1+\frac{1}{N}\right) \textgoth{Q}_{\eta^N_n}\quad \mbox{\rm with}\quad
\eta^N_n:=\frac{1}{N+1}\sum_{1\leq i\leq N+1}\delta_{\xi^i_n}.
$$
We also consider the sample means
\begin{equation}\label{sm-kalman-EnKF}
 m_n:=\frac{1}{N+1}\sum_{1\leq i\leq N+1}\xi^i_n
\quad\mbox{\rm and}\quad \widehat{m}_n:=\frac{1}{N+1}\sum_{1\leq i\leq N+1}\widehat{\xi}^i_n
\end{equation}
and the $\sigma$-field filtrations
$$
\begin{array}{l}
\Ga_n:=\sigma(\xi_0,\Wa_k,\Va_k, 0\leq k\leq n)\supset \sigma(\xi_k,\widehat{\xi}_k, ~0\leq k\leq n),\\
\\
\widehat{\Fa}_n:=\Ya_{n}\vee \Ga_n
\supset\Fa_n:=\widehat{\Fa}_{n-1}\vee \sigma(\Wa_n)\supset \widehat{\Fa}_{n-1}\vee\sigma(\xi_n).
\end{array}
$$

Note that this implies that
$$
\widehat{\Fa}_n\supset \Ya_{n}\quad \mbox{\rm and}\quad \Fa_n\supset \Ya_{n-1}.
$$

In the rest of the article we shall assume that $\xi_0$, $\Va_n$ and $\Wa_n$ are column vectors.
We also write $X~\underline{\in}~\Fa$ when a random variable is $\Fa$-measurable. With this notation, 
at every time $n\geq 0$, the updating-prediction steps of the ensemble Kalman filter  are described by the following synthetic diagram
$$
(m_n,p_n)~\underline{\in}~ \Fa_n\longrightarrow (\widehat{m}_{n},\widehat{p}_{n})~\underline{\in}~ \widehat{\Fa}_n\longrightarrow (m_{n+1},p_{n+1})~\underline{\in}~ \Fa_{n+1}.
$$
\subsection{Some basic notation}
This section presents some basic notation and preliminary results necessary for the statement of our main results. 

Throughout the rest of the article, we write $\tau$, $\tau_k$ as well as $\tau(l)$ and $\tau_k(l)$ for some collection of  non-universal constants whose values may vary from line to line, 
but only depend on the parameters $k$ and $l$ in some given parameter spaces. We use the letters $\iota$, $\iota_k$, $\iota(l)$, $\iota_k(l)$ 
as well as $\epsilon$ and $\epsilon_k\in ]0,1]$  when these constants may also depend
on the parameters $(A,B,C,D)$ of the filtering model.  
 Importantly these constants do not depend on the time parameter nor on the number of particles.

We denote by $\chi^{2}_{N ,Nx}$  a collection a non-central $\chi$-square random variables indexed by the degree of freedom $N\geq 1$ and non centrality parameter $N x\geq 0$. We recall that for any $N \ge 1$, we have
$$
\begin{array}{l}
\displaystyle 
\frac{1}{N}~\chi^2_{N,Nx}=\frac{1}{N}~\sum_{1\leq i\leq N}~\left(Z_i+\sqrt{x}\right)^2=1+x+
\sqrt{x}~\frac{2}{N}~\sum_{1\leq i\leq N}~Z_i+\frac{\sqrt{2}}{N}~\sum_{1\leq i\leq N}~\frac{Z_i^2-1}{\sqrt{2}},
\end{array}
$$
 for some given sequence $(Z_i)_{i\geq 1}$ of independent and centered Gaussian random variables. Rewritten in a slightly different form,  for any  $u,v$ with $v\not=0$  we have
\begin{eqnarray}
\frac{v}{N}~\chi^2_{N,N(u/v)^2}&=&
\displaystyle v+\frac{u^2}{v}+
\frac{2}{\sqrt{N}}~u~ \widecheck{\Za}+\sqrt{\frac{2}{N}}~v~ \widetilde{\Za},
\label{p-noncentral-chi}
\end{eqnarray}
where
 $$
 \widecheck{\Za}:=\frac{1}{\sqrt{N}}~\sum_{1\leq i\leq N}~Z_i\quad \mbox{\rm and}\quad
 \widetilde{\Za}:=\frac{1}{\sqrt{N}}~\sum_{1\leq i\leq N}~\frac{Z_i^2-1}{\sqrt{2}}.
$$
Note that these definitions imply that $\widecheck{\Za}$ is a centred Gaussian random variable with unit variance and $\widetilde{\Za}$ is a centered $\chi$-square variable. We also consider the extended random vector
\begin{equation}\label{def-Delta}
\Delta:=\left(
\begin{array}{c}
 \widehat{Z}\\
  \widecheck{\Za}\\
   \widetilde{\Za}
\end{array}
\right)\quad\mbox{\rm with}\quad  \widehat{Z}:=\frac{1}{\sqrt{N+1}}~\sum_{1\leq i\leq N+1}~Z_{N+i}.
\end{equation}
In a similar manner, we define $(\widehat{\Delta}_{n},\Delta_{n+1})$ as the sequence of random vectors
\begin{equation}\label{def-Delta-n}
\widehat{\Delta}_{n}:=\left(
\begin{array}{c}
\widehat\Delta^1_n\\
\widehat\Delta^2_n\\
\widehat\Delta^3_n
\end{array}
\right)=\left(
\begin{array}{l}
\widehat{V}_{n}\\
 \widecheck{\Va}_{n}\\
\displaystyle\widetilde{\Va}_n\end{array}\right)
\quad
\mbox{\rm and}
\quad
\Delta_{n+1}:=\left(
\begin{array}{c}
\Delta_{n+1}^1\\
\Delta_{n+1}^2\\
\Delta_{n+1}^3
\end{array}
\right)=
\left(\begin{array}{l}
\widehat{W}_{n+1}\\
 \widecheck{\Wa}_{n+1}\\
\displaystyle\widetilde{\Wa}_{n+1}
\end{array}\right)
\end{equation}
 indexed by the time parameter $n\geq 0$ and defined as $\Delta$ by replacing the sequence $(Z_i)_{i\geq 1}$ by the sequences of Gaussian variables
$\Va_n=(\Va^i_n)_{i\geq 1}$  and $\Wa_{n+1}=(\Wa^i_{n+1})_{i\geq 1}$ introduced in in the definition of the EnKF (\ref{kalman-EnKF-def}). For $n=0$ we also set
$$
\Delta_0:=\left(
\begin{array}{c}
\Delta^1_0\\
\Delta^2_0
\end{array}
\right)=\left(
\begin{array}{c}
\widehat{W}_0\\
\widetilde{\Wa}_0
\end{array}
\right)
$$
with $\widehat{W}_0$ and $\widetilde{\Wa}_0$ defined as $\widehat{Z}$ and $\widetilde{\Za}$ by replacing the sequence $Z_i$ with the sequence of Gaussian variables
$$
\Wa_0^i:=\frac{\xi^i_0-\widehat{X}^-_0}{\sqrt{P_0}},
$$
where we recall that the $\xi^i_0$ are independent copies of the Gaussian initial condition $X_0$ of the signal used to initialise the EnKF in (\ref{kalman-EnKF-def}). 
In this notation, up to a sequence of Helmert orthogonal  transformations on the sequence of Gaussian vectors $\Wa_0^i$, the initial condition of the ensemble Kalman filter is given by the stochastic perturbation formulae
\begin{equation}\label{def-initial-condition}
m_0=\widehat{X}^-_0+\frac{1}{\sqrt{N+1}}~\upsilon_0
\quad
\mbox{and}
\quad
p_0=P_0+\frac{1}{\sqrt{N}}~\nu_0,
\end{equation}
with the initial local perturbations defined by 
\begin{gather*}
\upsilon_0:=\sqrt{P_0}~\Delta^1_0\quad \mbox{and}\quad
\nu_0:=\sqrt{2}~P_0~\Delta^2_0.
\end{gather*}
Observe that the sample means are written in terms of random perturbations $\upsilon_0$ which are independent of the perturbations $\nu_0$ of the sample variances. The independence between the complete sufficient statistic $m_0$ and the ancillary statistic $p_0$ can also be checked using Basu's theorem~\cite{basu}. An alternative algebraic and direct approach based on Helmert orthogonal  matrices can be conducted using the same arguments as those used in the proof of Theorem~\ref{stoch-perturbation-theo}.

\section{Statement of the main results}\label{statement-sec}
\subsection{Local perturbation theorems}

 Consider the  local stochastic perturbations defined for any $n\geq 0$ by the formulae
\begin{gather}\label{def-loc-fluctuation}
\left\{
\begin{array}{rcl}
\widehat{\upsilon}_n&:=&-g_nD~\widehat{\Delta}^1_n\\
&&\\
\widehat{\nu}_n&:=&-2g_nD(1-g_nC)~\sqrt{p_n}~\widehat{\Delta}^2_n+\sqrt{2}~g_n^2D^2~\widehat{\Delta}^3_n
\end{array}\right.~
\left\{
\begin{array}{rcl}
\upsilon_{n+1}&:=&B~\Delta^1_{n+1}\\
&&\\
\nu_{n+1}&:=&2AB~\sqrt{\widehat{p}_n}~\Delta^{2}_{n+1}+\sqrt{2}~R~\Delta^3_{n+1}.
\end{array}\right.
\end{gather}
In the above display, $(\Delta_n,\widehat{\Delta}_n)_{n\geq 0}$ are the random vectors defined in (\ref{def-Delta-n}), and $g_n$ is the particle gain defined in (\ref{kalman-EnKF-def}).

\begin{theo}\label{stoch-perturbation-theo}
Up to a sequence of orthogonal transformations on the sequence of Gaussian vectors $\Wa_n$ and $\Va_n$, the ensemble Kalman filter and the Riccati equation are given by the initial condition (\ref{def-initial-condition}) and the updating-prediction transitions for $n \ge 0$
\begin{gather*}
\left\{
\begin{array}{rcl}
\displaystyle \widehat{m}_{n}&=&\displaystyle m_n+g_n~(Y_n-Cm_n)+\frac{1}{\sqrt{N+1}}~\widehat{\upsilon}_n\\
&&\\
\widehat{p}_n&=&\displaystyle(1-g_nC)~p_n+\frac{1}{\sqrt{N}}~ \widehat{\nu}_n
\end{array}
\right.~
\left\{
\begin{array}{rcl}
\displaystyle m_{n+1}&=&\displaystyle A~\widehat{m}_{n}+\frac{1}{\sqrt{N+1}}~\upsilon_{n+1}\\
&&\\
p_{n+1}&=&\displaystyle  A^2~\widehat{p}_{n}+R+\frac{1}{\sqrt{N}}~\nu_{n+1}.
\end{array}
\right.
\end{gather*}
\end{theo}

The proof of the above theorem is provided in section~\ref{proof-theo-stoch-intro}.
Merging the updating and prediction steps
 we check the following corollary.
 \begin{cor}\label{cor-srre}
The  sample variance $p_n$ of the EnKF satisfies the  stochastic Riccati rational difference equation given by the formula
\begin{equation}\label{ref-eq-p-Markov}
p_{n+1}=\phi(p_n)+\frac{1}{\sqrt{N}}~\delta_{n+1}\quad \mbox{with}\quad \delta_{n+1}:=A^2~\widehat{\nu}_n+\nu_{n+1}.
\end{equation}
In the above display, $\phi(p)$ is the one step mapping of the Riccati equation defined in (\ref{def-Ricc}), and the random variables $(\widehat{\nu}_n,\nu_{n+1})$ on the right hand side of (\ref{ref-eq-p-Markov}) are the local fluctuations defined in \eqref{def-loc-fluctuation}.
\end{cor}

Let $\Ra_n:=\sigma(p_k,~0\leq k\leq n)\vee \Ya_{n}$  be the $\sigma$-field generated by the stochastic Riccati equation $p_k$ and the observation sequence up to the time horizon $k\leq n$ and set
$$
\widehat{m}_{n}^r:=
\EE(\widehat{m}_{n}~|~\Ra_{n})\quad \mbox{\rm and}\quad m_{n+1}^r:= \EE(m_{n+1}~|~\Ra_{n}).
$$
Using this notation, the next corollary is a direct consequence of Theorem~\ref{stoch-perturbation-theo}.
 \begin{cor}
 Given the sample variances and the sequence of observations, for any $n \ge 0$,
 the updating-prediction steps of the sample mean of the EnKF satisfy the quenched Kalman-filter equations
$$
 \widehat{m}_{n}^r =m_n^r+ g_n(Y_n-C\,m_n^r)\quad \mbox{and}\quad
m_{n+1}^r=\displaystyle A~\widehat{m}_{n}^r,
$$
with the initial condition $m_{0}^r=\widehat{X}^-_0$.
  \end{cor}

The next corollary provides an alternative interpretation of the updating-prediction steps in terms of a two-steps Markov chain involving non-central $\chi$-square random variables.

\begin{cor}\label{cor-p-chi-evol}
The updating-prediction steps of the stochastic Riccati difference equations stated in Theorem~\ref{stoch-perturbation-theo} take the following form
\begin{equation}\label{wp-p-chi}
\widehat{p}_n=\left(\frac{p_n}{1+Sp_n}\right)^2~\frac{S}{N}~\widehat{\chi}^{(n,2)}_{N,N/(Sp_n)}\quad \mbox{and}\quad p_{n+1}=\frac{R}{N}~\chi^{(n+1,2)}_{N,N(A^2/R)\widehat{p}_{n}},
\end{equation}
with the initial condition
$$
p_0=\frac{P_0}{N}~\chi^{(0,2)}_{N,0}.
$$
In the above display, $\chi^{(n,2)}_{N,x}$ and $\widehat{\chi}^{(n,2)}_{N,x}$ stand for a collection of independent non-central $\chi$-square random variables with $N$ degrees of freedom
and non-centrality parameter $x$, which are independent of the local perturbation sequence $(\upsilon_{k+1},\widehat{\upsilon}_k)$ of the EnKF filter defined in Theorem~\ref{stoch-perturbation-theo}.
\end{cor}

 The proof of the above corollary is a direct consequence of the decomposition (\ref{p-noncentral-chi}), thus it is left as an exercise for the reader.

\subsection{Stability theorems}\label{sec-stability-intro}
We denote by  $\phi^n=\phi^{n-1}\circ\phi$  the composition evolution semigroup of the Riccati difference equation (\ref{def-Ricc}). We also let $\widehat{X}_n(x,p)$ be the solution of the  Kalman filter
associated with the initial values $(\widehat{X}_0,P_0)=(x,p)$. The next theorem summarises the main stability properties of the Kalman filter and the Riccati difference equations.

\begin{theo}\label{theo-ricc-Kalman-intro}
The Riccati  equation (\ref{def-Ricc})  has an unique positive fixed point $P_{\infty}=\phi(P_{\infty})>0$ and for any $n\geq 0$ and any $p=(p_1,p_2)\in \RR_+^2$, we have the exponential contraction estimates
\begin{gather}\label{ricc-lower-cv-2}
 \vert \phi^n(p_1)-\phi^n(p_2)\vert\leq \iota_1~\left(1-\epsilon_1\right)^n~\vert p_1-p_2\vert.
\end{gather}
In addition, there exists some integer $k(p_1)\geq 1$ such that for any $x=(x_1,x_2)\in\RR^2$ we have 
\begin{equation}\label{p1-p2-q-2}
\EE\left(\left(\widehat{X}_n(x_1,p_1)- 
\widehat{X}_n(x_2,p_2)\right)^2\right)^{1/2}\leq~\iota(x,p)~(1-\epsilon_2)^{n-k(p_1)}~\left(\vert p_1-p_2\vert+\vert x_1-x_2\vert\right).
\end{equation}

\end{theo}
The proof of the above theorem is provided in section~\ref{theo-ricc-Kalman-intro-proof}.

Let $\Pa(p, dq)$ denote the Markov transitions associated with the stochastic Riccati Markov chain $p_n$ discussed in (\ref{ref-eq-p-Markov}). Given some non-negative function $\Ua$ on $\RR_+ := [0, \infty[$ we define the $\Ua$-norm of some function $f: \RR_+ \to \RR$ and some locally finite signed measure $\mu(dp)$ on $\RR_+$  by
\begin{equation}\label{U-norm}
\Vert f\Vert_{\tiny \Ua}:=\sup_{p\geq 0}\left\vert \frac{f(p)}{\Ua(p)+1/2}\right\vert\quad \mbox{\rm and}\quad \Vert \mu\Vert_{\tiny \Ua}:=\sup{\{\vert \mu(f)\vert~:~f~\mbox{\rm s.t.}~\Vert f\Vert_{\tiny \Ua}\leq 1\}}.
\end{equation}
The $\Ua$-Dobrushin ergodic coefficient of $\Pa$ is defined by
\begin{equation}\label{U-Dob}
\beta_{\Ua}(\Pa) = \sup_{(p,q) \in \RR_+^2}\frac{\Vert \Pa(p, \cdot) - \Pa(q, \cdot) \Vert_{\Ua}}{1 + \Ua(p) + \Ua(q)}.
\end{equation}
The next theorem concerns the stability properties of the stochastic Riccati rational difference equation 
presented in corollary~\ref{cor-srre}.
\begin{theo}\label{stab-riccati-stoch}
There exists an unique invariant measure $\pi$ such that $\pi = \pi \Pa$ and some 
Lyapunov function of the form $\Ua(p)=u p+v$ such that $\beta_{\Ua}(\Pa) < 1$ for some $u,v>0$. In addition, for any probability measures $\mu_1$ and $\mu_2$ on $\RR_+$ and for any $n\geq 1$ we have
$$
\Vert \mu_1\Pa^n - \mu_2\Pa^n\Vert_{\Ua} \le \beta_{\Ua}(\Pa)^n\Vert \mu_1 - \mu_2\Vert_{\Ua}.
$$
Moreover, for any function $f$ such that $|f(p)| \le 1/2 + \Ua(p)$, for any $p \in \RR_+$, we have
$$
|\Pa^n(f)(p) - \pi(f)| \le \beta_{\Ua}(\Pa)^n~(1 + \Ua(p) + \pi(\Ua)).
$$
\end{theo}

The above theorem comes from the fact that 
the Riccati map $\phi(p)$ is a uniformly bounded function (cf. for instance (\ref{ul-bounds}) and (\ref{ricc-lower-cv})). In addition, using (\ref{ref-eq-p-Markov}) or (\ref{wp-p-chi}) we readily check that $\Pa(p,dq)$ has a continuous density $k(p,q)>0$ with respect to the variables $(p,q)$ and the Lebesgue measure $dq$ on $\RR_+$.
It is rather well-known that these two properties ensure the existence of a Lyapunov function of the form $\Ua(u,v)=u p+v$ such that $\beta_{\Ua}(\Pa) < 1$. For the convenience of the reader a detailed proof of the above theorem is provided in the appendix.  The above theorem is an extension of the stability theorem for stochastic Riccati diffusions presented in the article~\cite{bishop-19} (see also section 2.2 and Theorem 2.4 in~\cite{BishopDelMoralMatricRicc} for the multivariate case) to the stochastic Riccati rational equations presented in Corollary~\ref{ref-eq-p-Markov}.

We now turn to quantifying the stability properties of the EnKF sample means.
  Since the one-step optimal predictor $\widehat{X}_n^-$ is a minimal variance estimate of $X_n$ we already know that
  $$
  M_n:=(m_n-X_n)\Longrightarrow
P_n=  \EE((\widehat{X}_n^--X_n)^2)\leq   \EE(M_n^2).
  $$ 
Thus for small sample sizes we expect $\EE((m_n-X_n)^2)$ to be much larger than $P_n$. To find some useful quantitative estimate, 
observe that
\begin{equation}\label{ref-key-obs}
\displaystyle M_{n+1}=\displaystyle \frac{A}{1+Sp_n}~M_n+\Upsilon_{n+1},
\end{equation}
with the conditionally centered Gaussian random variables
$$
\Upsilon_{n+1}:=Ag_nDV_n-BW_{n+1}+\frac{1}{\sqrt{N+1}}~\left(A~\widehat{\upsilon}_n+\upsilon_{n+1}\right).
$$
The decomposition (\ref{ref-key-obs}) shows that the global stability of the stochastic process $M_n$
depends on the long-time behavior of the random products defined for any $0\leq l\leq n$ by the formula
$$
\Ea_{l,n}:=\prod_{l\leq k\leq n}\frac{A}{1+Sp_k}.
$$
For stable signal drifts, i.e. when
$\vert A\vert <1$, the exponential decays of these random products is immediate. In the unstable case, that is when $\vert A\vert>1$,
none of the uniform estimates stated later in Theorem~\ref{theo-u-p} appear to be useful to quantify directly these random products.

\begin{theo}\label{key-theo}
For any $k\geq 1$ there exist some parameters $N_k\geq 1$ and $\epsilon_k\in ]0,1]$ such that for any $N\geq N_k$ and any time horizon $n\geq l\geq 1$ we have the uniform almost sure exponential decays and the time-uniform estimates 
\begin{equation}\label{expo-decay-p}
\EE\left(\left\vert \Ea_{l,n}\right\vert^k ~|~p_{l-1}\right)\leq  \iota_k~\left(1-\epsilon_k\right)^{n-l}\quad \mbox{and}\quad
\EE\left(\vert M_{n}\vert^k\right)\leq \iota_k.
\end{equation}
\end{theo}
The proof of this theorem is provided at the end of section~\ref{change-probab-sec}.

\bigskip

Let $\Ma((p,x), d(q,z))$ denote the Markov transitions associated with the coupled 
Markov chain $(p_n,M_n)$ defined by  formulae (\ref{ref-eq-p-Markov}) and (\ref{ref-key-obs}). The $\Va$-norms and the $\Va$-Dobrushin ergodic coefficient  $\beta_{\Va}(\Ma)$ of $\Ma$ associated with some non-negative function $\Va$ on $\RR_+\times\RR$ are defined as in (\ref{U-norm}) and (\ref{U-Dob}) by replacing the state space $\RR_+$ by $\RR_+\times\RR$ and the function $\Ua$ by $\Va$. For any $\epsilon\in ]0,1]$ there exists some time horizon $k\geq 1$ such that
 \begin{equation}\label{Lyap-form}
 \Va(p,x)= p+\vert x\vert\Longrightarrow
\Ma^k(\Va)(p,x)\leq \epsilon~ \Va(p,x)+\iota.
\end{equation}
A proof of the above Lyapunov inequality is provided in section~\ref{theo-enkf-wX-proof}.
 Arguing as in the proof of Theorem~\ref{stab-riccati-stoch} we readily prove the following theorem.
 
\begin{theo}\label{stab-riccati+M-stoch}
There exists a unique invariant measure $\mu$ such that $\mu= \mu \Ma$,
a Lyapunov function of the form $\Va(p,x)=u p+v \vert x\vert+w$ for some $u,v,w>0$, and some integer $k\geq 1$ such that $\beta_{\Va}(\Ma^k) < 1$.
In addition,  for any probability measures $\eta_1$ and $\eta_2$ on $\RR_+\times \RR$ and any $n\geq 1$ we have
$$
\Vert \eta_1\Ma^{nk} - \eta_2\Ma^{nk}\Vert_{\Va} \le \beta_{\Va}(\Ma^k)^n\Vert \eta_1 - \eta_2\Vert_{\Va}.
$$
Moreover, for any function $f$ on $\RR_+\times \RR$ such that $|f(p,x)| \le 1/2 + \Va(p,x)$, for any $(p,x) \in \RR_+\times\RR$, we have
$$
|\Ma^{nk}(f)(p,x) - \mu(f)| \le \beta_{\Va}(\Ma^k)^n~(1 + \Va(p,x) + \mu(\Va)).
$$
\end{theo}

\subsection{Uniform mean-error estimates}\label{sec-unif-intro}

The next theorem provides uniform mean-error and bias estimates for  the particle gain defined in (\ref{kalman-EnKF-def}) and the stochastic Riccati equations
presented in Theorem~\ref{stoch-perturbation-theo}.
\begin{theo}\label{theo-u-p}
  For any $k\geq 1$ and any $N\geq 1$ and $n\geq 0$ we have the   time uniform estimates
\begin{equation}
\EE\left(\vert p_n-P_n\vert^k\right)^{1/k}\vee \EE\left(\vert \widehat{p}_n-\widehat{P}_n\vert^k\right)^{1/k}\vee \EE\left(\vert g_n-G_n\vert^k\right)^{1/k} \leq \iota_k~(1\vee P_0) /\sqrt{N}.
\label{time-unif-bound}
\end{equation}
In addition, we have
the uniform bias estimate
\begin{equation}\label{bias-estimates}
0\leq P_n-\EE(p_n)\leq \frac{\iota_1}{N}~\left[1\vee P_0\right]^{2}\quad\mbox{and}\quad
0\leq G_n-\EE(g_n)\leq  \frac{\iota_2}{N}~\left[1\vee P_0\right]^{2}.
\end{equation}
The same formula on the l.h.s. of the above display holds by replacing $(p_n,P_n)$ by $(\widehat{p}_n,\widehat{P}_n)$ .
\end{theo}
The proof of the uniform mean error estimates stated in the above theorem is provided in section~\ref{mean-error-P-sec}. The proof of the bias estimates is provided in the end of section~\ref{theo-loc-tcl-proof}.

Now we turn to quantifying the fluctuations of the sample means around the optimal filter.
We  set
$$
\widehat{m}_{n}^o:=
\EE(\widehat{m}_{n}~|~\Ya_{n})\quad \mbox{\rm and}\quad m_{n+1}^o:= \EE(m_{n+1}~|~\Ya_{n}).
$$
The next theorem provides uniform mean-error and bias estimates for  the sample means of the EnKF filter defined in (\ref{sm-kalman-EnKF}).
\begin{theo}\label{theo-enkf-wX}
For any $k\geq 1$ there exists some  parameter $N_k\geq 1$ such that for any $N\geq N_k$  and $n\geq 0$ we have the time-uniform estimates
\begin{equation}\label{estimates-m-wX-k}
\EE\left(\vert m_{n}-\widehat{X}^-_{n}\vert^k\right)^{1/k}\vee \EE\left(\vert \widehat{m}_{n}-\widehat{X}_{n}\vert^k\right)^{1/k}\leq {\iota_k}/{\sqrt{N}}.
\end{equation}
In addition, we have  the time-uniform bias estimates
\begin{equation}\label{estimates-bias-m-wX-k}
\EE\left(\vert m_{n}^o-\widehat{X}^-_{n}\vert^k\right)^{1/k}\vee \EE\left(\vert \widehat{m}^o_{n}-\widehat{X}_{n}\vert^k\right)^{1/k}\leq \iota_k~(1\vee P_0)^2/N.
\end{equation}
\end{theo}
The proof of the mean error estimates (\ref{estimates-m-wX-k}) is provided in the beginning of section~\ref{theo-enkf-wX-proof}. The uniform bias estimates (\ref{estimates-bias-m-wX-k}) are  proved in the end of section~\ref{theo-enkf-wX-proof}.

\subsection{Central limit theorems}\label{sec-clt-intro}

In the further development of this section, $Z^i_k$ and $\widehat{Z}^i_k$ stand for a sequence of independent, centered Gaussian random variables with unit variance indexed by $i\geq 1$, $k\geq 0$. We also consider the variables
$$
(\UU_0,\VV_0):=(\sqrt{P_0}~Z^1_0,\sqrt{2}~P_0~Z^2_0)
$$
and the sequence of the Gaussian random variables 
\begin{gather*}
\left\{\begin{array}{rcl}
\widehat{\UU}_n&:=&\displaystyle -G_nD~\widehat{Z}^1_n\\
&&\\
\widehat{\VV}_n&:=&\displaystyle -2G_nD(1-G_nC)~\sqrt{P_n}~\widehat{Z}^2_n+\sqrt{2}~G_n^2D^2~\widehat{Z}^3_n
\end{array}\right.\left\{
\begin{array}{rcl}
\UU_{n+1}&:=&\displaystyle B~Z^1_{n+1}\\
&&\\
\VV_{n+1}&:=&\displaystyle 2AB~\sqrt{\widehat{P}_n}~Z^{2}_{n+1}+\sqrt{2}~R~Z^3_{n+1}.
\end{array}\right.
\end{gather*}
Observe that the above sequence of independent and centered Gaussian variables are defined in a similar manner to the local stochastic perturbations (\ref{def-loc-fluctuation}) by replacing the sample variances by their limiting values and 
the random vectors $(\widehat{\Delta}_n,\Delta_{n+1})$ by the Gaussian random variables $(\widehat{Z}_n,Z_{n+1})$.
\begin{theo}\label{theo-loc-tcl}
 For any time horizon $n\geq 0$ we have the weak convergence
$$
(\upsilon_0,\nu_0,(\widehat{\upsilon}_{k},\widehat{\nu}_{k},\upsilon_{k+1},\nu_{k+1})_{0\leq k\leq n})\hooklongrightarrow_{N\rightarrow\infty}~(\UU_0,\VV_0,(\widehat{\UU}_{k},\widehat{\VV}_{k},\UU_{k+1},\VV_{k+1})_{0\leq k\leq n}).
$$

\end{theo}
The proof of the above theorem is provided in section~\ref{theo-loc-tcl-proof}.
Applying the continuous mapping theorem we obtain the following corollary.
\begin{cor}\label{cor-Z}
 For any time horizon $n\geq 0$ we have the weak convergence
$$
(\nu_0,(\delta_k)_{1\leq k\leq n})\hooklongrightarrow_{N\rightarrow\infty}~(\ZZ_0,(\ZZ_k)_{1\leq k\leq n}),
$$
with the independent Gaussian random variables $\ZZ_n$ defined by
$$
 \ZZ_{n+1}:=A^2~\widehat{\VV}_{n}
+\VV_{n+1}\quad
\mbox{and}
\quad \ZZ_{0}=\VV_0.
$$
In the above display, $\nu_0$ and $\delta_k$ are the local fluctuations of the stochastic Riccati equation discussed  in (\ref{ref-eq-p-Markov}).\end{cor}

We now consider the collection of stochastic processes $(\widehat{\QQ}_n^N,\QQ^N_{n+1})$ indexed by $N\geq 1$ defined for any $n\geq 0$ by the updating-prediction synthetic diagram
$$
\QQ^N_{n}:=\sqrt{N}(p_{n}-P_{n})\longrightarrow
\widehat{\QQ}_n^N:=\sqrt{N}(\widehat{p}_n-\widehat{P}_n)
\longrightarrow
\QQ^N_{n+1}.
$$

\begin{theo}\label{theo-clt-p-intro}
The stochastic processes $(\widehat{\QQ}_n^N,\QQ^N_{n+1})$ converge in law in the sense of convergence of finite dimensional distributions, and as the number of particles $N\rightarrow\infty$, to a sequence of centered stochastic processes $(\widehat{\QQ}_n,\QQ_{n+1})$ with initial condition $\QQ_0 = \VV_0$ and updating-prediction transitions given by
\begin{equation}\label{recursion-QQ}
\left\{
\begin{array}{rcl}
\widehat{\QQ}_n&=&(1-G_nC)^{2}~\QQ_n+\widehat{\VV}_n\\
&&\\
\QQ_{n+1}&=&A^2\,\widehat{\QQ}_n+\VV_{n+1}.
\end{array}\right.
\end{equation}

\end{theo}
The proof of the above theorem is provided in section~\ref{theo-loc-tcl-proof}.

We consider the collection of stochastic processes $(\widehat{\XX}_n^N,\XX^N_{n+1})$ indexed by $N\geq 1$ and defined for any $n\geq 0$ by the updating-prediction synthetic diagram
$$
\XX^N_{n}:=\sqrt{N}(m_{n}-\widehat{X}^-_{n})\longrightarrow
\widehat{\XX}_n^N:=\sqrt{N}(\widehat{m}_n-\widehat{X}_n)
\longrightarrow
\XX^N_{n+1}.
$$

\begin{theo}\label{term-theo}
The stochastic processes $(\widehat{\XX}_n^N,\XX^N_{n+1})$ converge in law in the sense of convergence of finite dimensional distributions, as the number of particles $N\rightarrow\infty$, to a sequence of centered stochastic processes $(\widehat{\XX}_n,\XX_{n+1})$ with initial condition $\XX_0 = \UU_0$ and updating-prediction transitions given by

\begin{equation}\label{XX-w-XX-theo}
\left\{
\begin{array}{rcl}
\widehat{\XX}_{n}&=&(1-G_nC)\,\XX_n+\GG_n\,(Y_n-C\widehat{X}^-_n)+\widehat{\UU}_n\\
&&\\
\XX_{n+1}&=&A\,\widehat{\XX}_{n}+\UU_{n+1}.
\end{array}\right.
\end{equation}

\end{theo}

\subsection{Some comments and comparisons}

As shown in section~\ref{sec-EnKF-intro}, the EnKF  is a rather simple numerical filtering-type technique defined by an ensemble of particles mimicking the evolution of well-known 
Kalman filter equations, replacing `the true' covariances by the ensemble sample-covariances.  
Besides the fact that the EnKF is only consistent for linear Gaussian filtering problems (cf.~\cite{legland}), these popular particle methodologies are applied to complex nonlinear and high dimensional filtering problems arising in fluid mechanics~\cite{beyou,memin-1,memin-2}, weather forecasting~\cite{anderson-jl,anderson-jl-2,burgers,houte}, geosciences and reservoir simulation~\cite{evensen-reservoir,nydal,seiler,skj,weng}, and many other data assimilation disciplines.

To reduce the computational cost, the ensemble sample size is generally chosen 
to be several orders of magnitude below the very large effective dimension of the underlying signal. 
To prevent
the ensemble covariance degeneracy in time, small sample EnKF requires one to combine several customisation techniques such as adaptive covariance/gain inflation and related localisation methodologies~\cite{reich-1,reich-2,bishop-18,tong}. The first rigorous analysis of the bias and the consistency of these additional levels of approximations for linear Gaussian models can be found  in~\cite{bishop-18,bishop-20}. 
  
 As shown in~\cite{dm-k-tu}, for nonlinear filtering problems, the EnKF can be interpreted as particle-type extended Kalman filter.  
As any observer-type algorithm, the well-posedness is not connected to any kind of distance to the optimal filter for any small or large sample sizes, but in terms of non-degeneracy of the observer with respect to time. 
  
As underlined by Majda and Tong in~\cite{majda-2} ``Why EnKF works well with a small ensemble has remained a complete mystery.  Practitioners often attribute the EnKF success to the existence of a low effective filtering dimension. But the definition of the effective filtering dimension has remained elusive, as its associated subspace often evolves with the dynamics''. In the article~\cite{majda-2}, the authors suggest several theoretical guidelines to choose judiciously  the ensemble sample size using covariance inflation and spectral projection techniques on the unstable directions of the signal.

From a purely mathematical perspective, the first important question is clearly to find conditions that ensure that the EnKF is well-founded and consistent, in the sense that it converges to the optimal filter as the ensemble sample size, or equivalently the computational power, tends to infinity (cf. for instance~\cite{legland,mandel} in discrete time settings and~\cite{DelMoral/Tugaut:2016,dm-k-tu} for continuous time EnKF models).

Of course, the consistency property and asymptotic analysis of any statistical Monte Carlo estimate such as the sample means or the sample covariances delivered by EnKF is reassuring but it only gives a partial answer to the performance of these ensemble filters when used with small sample sizes.

The non-asymptotic  analysis for finite sample sizes is clearly a more delicate subject as the continuous or the discrete generation EnKF belong to unconventional classes of nonlinear Markov processes and mean-field particle processes interacting through the sample covariance of the system. 
Here the interaction covariance function of the EnKF  particle filter is not a bounded nor a Lipschitz function as is it is commonly assumed in traditional nonlinear and interacting diffusion theory. As a result, none of the stochastic tools developed in this field, including the
rather elementary variational methodology for nonlinear diffusions developed in~\cite{arnaudon-dm-2,arnaudon-dm-3} nor the more recent powerful differential calculus  developed in~\cite{delarue-1},  can be used to analyse this class of continuous-time models equipped with a 
nonlinear quadratic-type interacting function.  Here, the sample covariance matrices  satisfy a rather sophisticated quadratic-type stochastic Riccati  matrix
diffusion equation, which requires one to develop new stochastic analysis tools~\cite{BishopDelMoralMatricRicc}. 
For a more thorough discussion on mean-field type particle systems and a probabilistic description of genetic type particle filters and Ensemble Kalman type particle filters within this framework we refer to the review article~\cite{bishop-20}, the books~\cite{dm-13,DelMoral/Tugaut:2016,dm-penev} and references therein.

The analysis of the discrete generation EnKF discussed in the present article is more delicate and requires the development of new stochastic methodologies. For instance, in discrete time settings, the EnKF particle system evolves as the Kalman filter (\ref{kalman-rec}) in terms of a two-step prediction-updating interacting Markov chain  (\ref{kalman-EnKF-def}). In contrast with the Gaussian nature of the continuous time EnKF and stochastic Riccati diffusions discussed above,  Theorem~\ref{stoch-perturbation-theo} presented in this article shows that these two transitions combine both Gaussian and $\chi$-square type perturbations. To the best of our knowledge, the analysis of stochastic Riccati rational difference equations, such as those discussed in Corollary~\ref{cor-srre}, has not been covered in the literature. 
Corollary~\ref{cor-p-chi-evol} also provides an alternative description of the sample variances in terms of a Markov chain with  non-central $\chi$-square nonlinear fluctuations.
Again, to the best of our knowledge, these stochastic perturbation theorems and the stability analysis of the analysis of stochastic Riccati rational difference equations are the first of this type in the applied probability literature. 

In a study by Tong, Majda and Kelly~\cite{tong}, the authors analyse the long-time behaviour and ergodicity of discrete generation EnKF using Foster-Lyapunov techniques ensuring that the filter is asymptotically stable with respect to any erroneous initial condition as soon as the signal is stable.
These important properties ensure that the EnKF has a single invariant measure 
and initialisation errors of the EnKF will dissipate with respect to the time parameter.  

Nevertheless, the only ergodicity of the signal and thus the ensemble process does not give any useful information on the convergence and the accuracy of 
the EnKF towards the optimal filter as the number of samples tends to infinity, nor does it allow one to quantify the fluctuations around the optimal filter for a given small or even very large ensemble sample size. On the other hand, effective unstable  directions are connected to unstable-transient signals that may grow exponentially fast. In this context, for obvious reasons, any well tracking EnKF needs to be an unstable-transient particle process. As a consequence, we cannot expect to obtain any type of filter ergodicity properties, nor any time-uniform estimates for the raw moments of the EnKF filter.

From our point of view, the terminology "effective filtering dimensions" discussed above should be understood as the unstable directions of the signal. For time varying and multivariate linear-Gaussian filtering problems, these effective dimensions are rarely known as they are partially observed and they may change in time. For a more detailed discussion on  the stability properties of systems of  time varying linear stochastic differential equations with random coefficients we refer to~\cite{stab-OU-19}.

To avoid the time degeneracy of any type of filtering approximation, a crucial question in practice  is to check wether or not the convergence is uniform with respect to the time parameter. Equivalently, we need to obtain estimates of {\em the error} between the optimal filter and its approximation that are uniform with respect to any time horizon {\em even if the signal and thus any well tracking approximating filter are both unstable and transient}. We emphasise that in this setting, any good approximating filter will diverge. In this context, the filtering step is part of a control loop that allows one to track the transient signal at all times.

The main advantage of one-dimensional filtering problems over multivariate problems with unstable signals is that they have a single effective dimension. 
In this context, the EnKF is also unstable and transient but the
 theorems~\ref{theo-u-p} and~\ref{theo-enkf-wX} presented in this article show that
 the EnKF tracks and converges to the optimal filter uniformly in time.
  These theorems also ensure that the precision error to the optimal filter can be quantified
  for any given finite sample size, and it will not degrade with respect to the time horizon. Theorem~\ref{stoch-perturbation-theo} also indicates that  the sample variances satisfy an autonomous stochastic Riccati-type evolution equation, independent of the sequence of observations.
 From a different angle, theorem~\ref{stab-riccati-stoch} also shows that the sample variances of the EnKF behaves as a Markov chain that converge exponentially fast as the time horizon tends to infinity to a single invariant measure, {\em even if the signal of the filtering problem is transient}.
 
   To the best of our knowledge, this stability theorem and the non-asymptotic theorems discussed above are the first results of this type in the literature on particle filter theory, including the literature of discrete generation EnKF.
   Next we provide a more detailed discussion on the stability of particle filters and their particle approximations, with precise reference pointers. We also suggest an avenue of open research problems.

Most of the stability properties of genetic-type particle filters~\cite{dm-miclo,fk}, including continuous time interacting jump Feynman-Kac particle systems~\cite{arnaudon-dm}, are only valid for stable and ergodic signals. It is clearly out the scope of this article to review in detail all contributions related to the stability of genetic type particle filters, thus we refer to the books~\cite{dm-miclo,dm-13,fk} and references therein. The analysis of these continuous-time or discrete generation genetic type processes for unstable signals remains an important and open research question.

Continuous time EnKF filters for stable signals are also discussed in~\cite{ap-2016,dm-k-tu}, including in the context of the extended Kalman filter. The analysis of unstable and possibly transient signals is more recent, starting at the end of the 2010s in the context of continuous time filtering problems~\cite{dm-b-niclas,BishopDelMoralMatricRicc,bishop-19}, see also~\cite{bishop-20} for a comprehensive review of this subject.  The articles~\cite{BishopDelMoralMatricRicc,dm-b-niclas} analyse the stability and the fluctuation properties of the sample covariance matrices of continuous time EnKF filters for multivariate models under natural and minimal observability and controllability conditions.

To better understand the difficulty in handling unstable effective signal directions, we recall that under natural observability and controllability conditions the Kalman filtering theory developed by Kalman and Bucy in the end of the 1960s ensures that the optimal filter is able to track any possibly unstable signals uniformly with respect to the time horizon. To answer any stability question related to EnKF particle filters or genetic type particle filters, we first need to extend the theory of Kalman and Bucy to this class of approximating particle filters. This research project remains open for genetic type filters. The present article provide a partial answer for discrete generation EnKF particle filters.

To better understand the difficulty in handling and extending  the stability theory of Kalman and Bucy to particle filters,
  it is worth noting a brief of history.

The stability properties of continuous time or discrete generation Kalman filters and their associated Riccati equations for multidimensional filtering problems are usually conducted  under some observability/controllability-type conditions using sophisticated and tedious matrix manipulations.

The first stability properties for continuous time models go back to the beginning of the 1960s with the celebrated work of Kalman and Bucy~\cite{kalman61}, with related prior work by Kalman \cite{kalman60,kalman60-2,kalman60-3} and later work by Bucy \cite{bucy2}. In \cite{anderson71}, the stability of this filter was analysed under a relaxed controllability condition.  The first stability properties for discrete time models go back to to the end of the 1960s with the work of Deyst and Price~\cite{deyst}, which was incorporated in the seminal lecture book of Jazwinski~\cite{jazwinski}.

 However, as noted in \cite{hitz72}, there was in both cases a crucial and commonly made error in the proof which invalidated the results, see~\cite{maybeck} and the more recent articles~\cite{ap-2016,rhudy} for a more detailed discussion and references on these issues. This error was repeated  in numerous subsequent works. A first correction \cite{bucy72corrected} was noted in a reply to \cite{hitz72}; see Bucy's reply \cite{bucy72remarks} and a separate reply by Kalman \cite{kalman72}. However, a complete reworking of the result did not appear in entirety until 2001 by Delyon in~\cite{delyon2001}.   
 
The stability of Kalman filter and Riccati equations is nowadays rather well understood.
Nevertheless it is not always simple to find a self-contained, rigorous and easy-to-read study on these subjects. For a complete and self-contained analysis on the stability and convergence of Kalman-Bucy filtering in continuous time settings, we also refer to~\cite{ap-2016}.  A  self-contained analysis of  Riccati differential equations and discrete-time 
Kalman filters for one-dimensional models is provided in section~\ref{sec-riccati}
and section~\ref{theo-ricc-Kalman-intro-proof} in the present article, see also Theorem~\ref{theo-ricc-Kalman-intro}.

The stability analysis of
continuous time stochastic Riccati  matrix diffusions is also  rather well understood, see for instance~\cite{ap-2016,BishopDelMoralMatricRicc,dm-b-niclas,bishop-20}. 
The extension of these multivariate results in the discrete time case remains an important open question. Theorem~\ref{theo-ricc-Kalman-intro} as well as (\ref{ref-key-obs}) and Theorem~\ref{key-theo}  provide a complete answer to the stability of discrete generation EnKF for one dimensional filtering problems, with possibly unstable signals, yielding what seems to be the first results of this type for this class of particle filters, including both EnKF and genetic type particle filters.

\section{Some preliminary results}
In this section we give several technical results that will be useful in the further development of the article. We start by considering the inverse moments of the non-central $\chi$-square random variables introduced in \eqref{p-noncentral-chi}. We then provide a self-contained analysis of the Riccati difference equations defined in \eqref{def-Ricc}.

\subsection{Non-central $\chi$-square moments}
In this section, we will prove the following technical result that allows one to estimate the inverse raw moments of non-central $\chi$-square random variables.

\begin{prop}\label{prop-raw-moments-chi}
For any $k\geq 0$ and  $n>(k+1)$ we have the uniform estimate
\begin{equation}\label{inv-moments-k}
0\leq 
\EE\left(\left(\frac{2n}{\chi^2_{2n,2nx}}\right)^{k}\right)-\frac{1}{\left(1+x\right)^{k}}\leq 
 \frac{k+1}{n}~k~\left(1+\frac{k+1}{n-(k+1)}\right)^{k}.
\end{equation}
In addition, there exists a constant $\omega_k$ such that, for any $n\geq 2(k+2)$ we have
\begin{equation}\label{inv-moments-k-3}
 \EE\left(\left(\frac{1+x}{\frac{1}{2n}\chi^2_{2n,2nx}}-1\right)^{k}\right)\leq  \frac{\omega_k}{n}~\left(1+x\right)^{k}.
 \end{equation}
\end{prop}

\medskip

In order to prove this result, let us first consider the following Taylor series expansion,
$$
e^{-x}=\sum_{0\leq k\leq n}\frac{(-x)^k}{k!}+\frac{(-x)^{n+1}}{(n+1)!}~\tau_{n+1}(x), n \ge 0,
$$
with the collection of integral remainders given by
$$
\tau_{n+1}(x)=(n+1)~\int_{0}^1~(1-t)^ne^{-tx}~dt=1-\frac{x}{n+2}~\tau_{n+2}(x).
$$
For any $k\geq 1$ we denote by  $\tau^{(k)}_n$ s the  $k$-th differential of the function $\tau_n$.
In this notation, 
 the inverse moment of a non-central $\chi$-square random variable (cf. section 3 in~\cite{bock}) are given, for any $k\geq 0$, by the formulae
$$
\begin{array}{rcl}
\displaystyle
\EE\left(\left(\frac{2}{\chi^2_{2(n+1),2x}}\right)^{(k+1)}\right)
&=&\displaystyle\frac{(-1)^{k}}{k!}~\frac{1}{n-k}~\tau^{(k)}_{n-k}(x).
\end{array}
$$

Setting $k=0$ we find the formulae
\begin{equation*}
\frac{1}{n}~\tau_{n}(x)
=\EE\left(\frac{2}{\chi^2_{2(n+1),2x}}\right)\geq\frac{2}{\EE(\chi^2_{2(n+1),2x})}=\frac{1}{x+n+1},
\end{equation*}
where we have used Jensen's inequality to obtain the lower bound. This combined with another application of Jensen's inequality implies that
$$
\begin{array}{rcl}
\displaystyle
\frac{(-1)^{k}}{k!}~\frac{1}{n-k}~\tau^{(k)}_{n-k}(x) = 
\EE\left(\left(\frac{2}{\chi^2_{2(n+1),2x}}\right)^{(k+1)}\right)
&\geq&\displaystyle \left(\frac{1}{n}~\tau_{n}(x)\right)^{k+1}\geq  \frac{1}{(x+n+1)^{k+1}}.
\end{array}
$$


On the other hand, note that
for any $0\leq z<1$ we have
$$
\log{(1-z)}\leq -z
\Longrightarrow
 n\log{\left(1-\frac{s}{m}\right)}-sx\leq -\left(x+\frac{n}{m}\right)~s.
$$
This yields the upper bound
$$
m\tau_{n}(mx) = \int_0^{m}~\left(1-\frac{s}{m}\right)^n~e^{-sx}~ds\leq \int_0^{m}~e^{-\left(x+\frac{n}{m}\right)s
}~ds=\frac{1}{x+\frac{n}{m}}~\left(1-e^{-\left(x+\frac{n}{m}\right)m
}\right)\leq \frac{1}{x+\frac{n}{m}}.
$$
Thus, 
for any $m,n\geq 1$ and any $x\geq 0$  we have the estimate
$$
 \frac{1}{x+(n+1)/m}\leq
\frac{m}{n}~\tau_{n}(mx)\leq \frac{1}{x+{(n-1)}/{m}}.
$$

More generally, for any $k,m,n\geq 0$ we have
\begin{eqnarray*}
\frac{(-1)^k}{n+1}~\tau_{n+1}^{(k)}(mx)&=&\int_0^1 t^k~(1-t)^n~e^{-mxt}~dt\\
&=&
\frac{1}{m}\int_0^{m}~\left(\frac{s}{m}\right)^k~\left(1-\frac{s}{m}\right)^n~e^{-sx}~ds\\
&\leq & \frac{1}{m^{k+1}}\int_0^{m}~s^k~e^{-\left(x+\frac{n}{m}\right)s
}~ds\leq   \frac{1}{m^{k+1}}~\frac{1}{\left(x+\frac{n}{m}\right)}~\frac{k!}{\left(x+\frac{n}{m}\right)^{k}}.
\end{eqnarray*}
This discussion can be summarised with the following lemma.

\begin{lem}\label{lem-tech-tau-kn}
For any $n> k\geq 0$, $m\geq 1$ and any $x\geq 0$ we have the estimate
\begin{eqnarray*}
\frac{1}{\left(x+\frac{n+1}{m}\right)^{k+1}}\leq 
\frac{(-1)^{k}}{k!}~\frac{m^{k+1}}{n-k}~\tau^{(k)}_{n-k}(mx)& \leq&  ~\frac{1}{\left(x+\frac{n-k-1}{m}\right)^{k+1}}.
\end{eqnarray*}
\end{lem}

\bigskip

With this lemma in hand, we are now in a position to prove Proposition \ref{prop-raw-moments-chi}.

\medskip

\begin{proof}[Proof of Proposition \ref{prop-raw-moments-chi}]
By Lemma~\ref{lem-tech-tau-kn} we have
\begin{eqnarray*}
\frac{1}{\left(x+\frac{n+1}{m}\right)^{k+1}}\leq 
\EE\left(\left(\frac{2m}{\chi^2_{2(n+1),2mx}}\right)^{(k+1)}\right)& \leq&  ~\frac{1}{\left(x+\frac{n-k-1}{m}\right)^{k+1}}.
\end{eqnarray*}
This yields the estimate
$$
\frac{1}{\left(1+x\right)^{k}}\leq 
\EE\left(\left(\frac{2n}{\chi^2_{2n,2nx}}\right)^{k}\right)\leq ~\frac{1}{\left((1+x)-\frac{k+1}{n}\right)^{k}}.
$$
This implies that for any $n> k$ we have
$$
0\leq 
\EE\left(\left(\frac{2n}{\chi^2_{2n,2nx}}\right)^{k}\right)-\frac{1}{\left(1+x\right)^{k}}\leq \theta_k(x) \leq \frac{k+1}{n}~\frac{k}{\left(x+\frac{n-(k+1)}{n}\right)^{k}\left(1+x\right)}
$$
with the function
$$
\theta_k(x):=~\frac{1}{\left((1+x)-\frac{k+1}{n}\right)^{k}}-\frac{1}{\left(1+x\right)^{k}}.
$$
This ends the proof of the first assertion.

This first estimate also yields, for any $n\geq k+2$, the rather crude estimate
$$
0\leq \EE\left(\left(\frac{1+x}{\frac{1}{2n}\chi^2_{2n,2nx}}\right)^{k}\right)-1\leq 
 \frac{k}{n}~(k+1)\left(1+x\right)^{k}\left(k+2\right)^{k}\leq  \frac{1}{n}~\left(1+x\right)^{k}\left(k+2\right)^{k+2}.
$$
This implies that
\begin{equation}\label{inv-moments-k-2}
0\leq \EE\left(\left(\frac{1+x}{\frac{1}{2n}\chi^2_{2n,2nx}}\right)^{k}\right)-1\leq  \frac{\omega_k}{n}~\left(1+x\right)^{k}\quad \mbox{\rm with}\quad \omega_k:=\left(k+2\right)^{k+2}.
\end{equation}
Using the decomposition
$$
\left(y-1\right)^k=\sum_{0\leq l\leq  k}\left(\begin{array}{c}
 k\\
 l
 \end{array}
 \right)~(-1)^{l}~y^{k-l}=\sum_{0\leq l< k}\left(\begin{array}{c}
 k\\
 l
 \end{array}
 \right)~(-1)^{l}~\left(y^{k-l}-1\right),
$$
we check that
$$
 \EE\left(\left(\frac{1+x}{\frac{1}{2n}\chi^2_{2n,2nx}}-1\right)^{k}\right)=\sum_{0\leq l< k}\left(\begin{array}{c}
 k\\
 l
 \end{array}
 \right)~(-1)^{l}~\underbrace{\left(\EE\left(\left(\frac{1+x}{\frac{1}{2n}\chi^2_{2n,2nx}}\right)^{k-l}\right)-1\right)}_{\geq 0}.
$$
Thus, for any $n\geq 2(2k+2)$ we have the estimate
\begin{eqnarray*}
 \EE\left(\left(\frac{1+x}{\frac{1}{2n}\chi^2_{2n,2nx}}-1\right)^{2k}\right)&\leq& \sum_{1\leq l\leq  k}\left(\begin{array}{c}
 2k\\
 2l
 \end{array}
 \right)~\left(\EE\left(\left(\frac{1+x}{\frac{1}{2n}\chi^2_{2n,2nx}}\right)^{2l}\right)-1\right).
\end{eqnarray*}
Using (\ref{inv-moments-k-2}), this implies that
$$
 \EE\left(\left(\frac{1+x}{\frac{1}{2n}\chi^2_{2n,2nx}}-1\right)^{2k}\right)\leq  \frac{1}{n}~\left(1+x\right)^{2k}\omega^{\prime}_{2k}.
$$
 with
 $$
  \omega^{\prime}_{2k}:=\sum_{1\leq l\leq  k}\left(\begin{array}{c}
 2k\\
 2l
 \end{array}
 \right)~\left(2(l+1)\right)^{2(l+1)}.
$$
In the same vein, for any $n\geq 2((2k+1)+2)$ we have
\begin{eqnarray*}
 \EE\left(\left(\frac{1+x}{\frac{1}{2n}\chi^2_{2n,2nx}}-1\right)^{2k+1}\right)&\leq& \sum_{0\leq 2l< 2k+1}\left(\begin{array}{c}
 2k+1\\
 2l
 \end{array}
 \right)~\underbrace{\left(\EE\left(\left(\frac{1+x}{\frac{1}{2n}\chi^2_{2n,2nx}}\right)^{2(k-l)+1}\right)-1\right)}_{\geq 0}\\&\leq& \sum_{0\leq l\leq  k}\left(\begin{array}{c}
 2k+1\\
 2(k-l)
 \end{array}
 \right)~\underbrace{\left(\EE\left(\left(\frac{1+x}{\frac{1}{2n}\chi^2_{2n,2nx}}\right)^{2l+1}\right)-1\right)}_{\leq \frac{\omega_{2l+1}}{n}~ (1+x)^{2l+1}}.
\end{eqnarray*}
This implies that
$$
 \EE\left(\left(\frac{1+x}{\frac{1}{2n}\chi^2_{2n,2nx}}-1\right)^{2k+1}\right)\leq \frac{\omega^{\prime}_{2k+1}}{n}~ (1+x)^{2k+1}
$$
with
$$
\omega^{\prime}_{2k+1}=\sum_{0\leq l\leq  k}\left(\begin{array}{c}
 2k+1\\
 2(k-l)
 \end{array}
 \right)~(2l+3)^{2l+3}.
$$
This ends the proof of the proposition.
\end{proof}

\subsection{Riccati difference equations }\label{sec-riccati}
We associate with some parameters $(a,b,c,d)\in \RR_+^4$ satisfying $c>0$ and $ad>bc$ the Riccati difference equation on $\RR_+:=[0,\infty[$ given by the recursion
\begin{equation}\label{ricc-eq}
P_n:=\phi(P_{n-1})\quad \mbox{\rm with}\quad \phi(x):=\frac{a x+b}{cx+d}\quad \mbox{\rm with}\quad x\geq 0.
\end{equation}
Observe that
\begin{equation}\label{def-rho}
\partial \phi(x)=\frac{\rho}{(cx+d)^2}> 0\quad \mbox{\rm and}\quad 
\partial^2 \phi(x)=-\frac{2\rho c}{(cx+d)^3}< 0\quad\mbox{\rm with}\quad \rho:=ad-bc>0.\\
\end{equation}
This shows that $\phi$ is an increasing and concave function.
Rational difference equation of the form (\ref{ricc-eq}) can be solved explicitly. This section provides a brief review on these equations. We set
$$
(u,v,w):=(a/c,d/c,b/c)\in \RR_+^4\quad\mbox{\rm and}\quad\lambda=\lambda_1/\lambda_2\in ]0,1[$$
with the parameters
\begin{equation}\label{parameters-ricc}
 \lambda_2-v:=\frac{u-v}{2}+\sqrt{\left(\frac{v-u}{2}\right)^2+w}>0 \quad \mbox{\rm and}\quad
v-\lambda_1:=\frac{v-u}{2}+\sqrt{\left(\frac{v-u}{2}\right)^2+w}> 0.
\end{equation}
Observe that
\begin{equation}\label{fixed-point-ricc}
(v-\lambda_1) (\lambda_2-v)=w \qquad \mbox{\rm and}\qquad r=\phi(r)>0\Longleftrightarrow r=\lambda_2-v.
\end{equation}
\begin{lem}\label{lem-ricc}
For any $x\geq 0$ and $n\geq 1$ the evolution semigroup $\phi^n=\phi\circ \phi^{n-1}$ of the Riccati difference equation (\ref{ricc-eq}) is positive and is given by the formula
\begin{equation}\label{sol-ricc}
\phi^n(x)=r+(x-r)~
\frac{\left(\lambda_2-\lambda_1\right)~\lambda^n}{\left(x+(v-\lambda_1)\right)~\left(1-\lambda^{n}\right)+(\lambda_2-\lambda_1)~\lambda^{n}}
\end{equation}
with the parameters $\lambda$ and $(\lambda_1,\lambda_2)$ defined in (\ref{parameters-ricc}) and the unique positive fixed point $r$ given in (\ref{fixed-point-ricc}). In addition, for any $n\geq 1$ and any $x\in\RR_+$ we have the estimates
\begin{equation}\label{ul-bounds}
{b}/{d}\leq \phi^n(x)\leq  {a}/{c}.
\end{equation}
For any $x,y\in \RR_+$ and $n\geq 0$ we have the exponential decay to equilibrium
\begin{equation}\label{expo-bounds}
\vert \phi^n(x)-r\vert \leq 
\frac{\lambda_2-\lambda_1}{v-\lambda_1}~\lambda^n~\vert x-r\vert \quad \mbox{and}\quad \vert \phi^n(x)-\phi^n(y)\vert\leq \vert x-y\vert~\left(\frac{\lambda_2-\lambda_1}{v-\lambda_1}\right)^2~\lambda^n.
\end{equation}
In addition, we have the first order derivative 
\begin{equation}\label{1st-order-exp}
\partial \phi^n(y):=\frac{(\lambda_2-\lambda_1)^2~\lambda^n}{\big(\left(y+(v-\lambda_1)\right)~\left(1-\lambda^{n}\right)+(\lambda_2-\lambda_1)~\lambda^{n}\big)^2}=\prod_{0\leq k<n}\frac{\rho}{(c\phi^{k}(y)+d)^2},
\end{equation}
where the parameter $\rho$ was defined in (\ref{def-rho}), which satisfies the second order estimate
\begin{equation}\label{2nd-order}
\left\vert \phi^n(x)-\phi^n(y)-\partial \phi^n(y) (x-y)\right\vert\leq \displaystyle \frac{(\lambda_2-\lambda_1)^2}{(v-\lambda_1)^3}~(x-y)^2~\lambda^n,
\end{equation}
and the Lipschitz estimate
\begin{equation}\label{Lip-2-ricc}
\begin{array}{l}
\displaystyle
\vert\partial \phi^n(x)-\partial \phi^n(y)\vert\leq \iota~
\vert y-x\vert~\lambda^n.
\end{array}
\end{equation}
\end{lem}

\bigskip

The proof of the above lemma is rather technical so it is housed in the appendix.

We end this short section with some comparison properties and Riccati differential inequalities. For any $\epsilon\geq 0$ we have
\begin{equation}\label{ricc-eq-epsilon}
\phi_{\epsilon}(x):=\phi(x)+\epsilon=\frac{a_{\epsilon}x+b_{\epsilon}}{cx+d}
\end{equation}
with the parameters
$$
(a_{\epsilon},b_{\epsilon})=((a+\epsilon c),(b+\epsilon d))\in \RR_+^2\Longrightarrow a_{\epsilon}d-b_{\epsilon}c=ad-bc>0.
$$
This shows that the boundedness properties of Riccati type differential 
inequalities of the form $r_n\leq \phi_{\epsilon}(r_{n-1})$  can be deduced from those of the solution to the evolution equation $s_n=\phi_{\epsilon}(s_{n-1})$. For instance, since $\phi_{\epsilon}$ is increasing we have
\begin{equation}\label{ricc-eq-epsilon-bound}
 \phi_{\epsilon}(r_{n-1})\leq \phi_{\epsilon}^n(x)\leq x\vee s^{\epsilon}\leq x\vee \frac{a_{\epsilon}}{c}= x\vee (\epsilon+\frac{a}{c})\quad \mbox{\rm with the fixed point}\quad
 s^{\epsilon}:= \phi_{\epsilon}(s^{\epsilon})>0.
\end{equation}
Moreover, observe that for any $\epsilon\leq \epsilon_1$ we have
\begin{eqnarray*}
0<2c~s^{\epsilon}&=&a-d+\epsilon c+\sqrt{\left((a-d)+\epsilon c\right)^2+4c~(b+\epsilon d)}\\
&\leq& 
a-d+\sqrt{2(a-d)^2+4bc}+2\sqrt{\epsilon_1c~\left(d+\frac{\epsilon_1 c}{2}\right)}+\epsilon_1 c.
\end{eqnarray*}
In the same vein, for any $k\geq 1$ we have
\begin{equation}\label{ricc-eq-power-k}
\phi(x)^k=\left(\frac{a x+b}{cx+d}\right)^k\leq \phi_k(x^k)\quad \mbox{\rm with}\quad \phi_k(x):=\frac{a_{k}x+b_{k}}{c_kx+d_k}
\end{equation}
and where
$$
(a_{k},b_{k},c_k,d_k):=(2^{k-1}a^k,2^{k-1}b,c^k,d^k)\in \RR_+^4.
$$
Note that we have
$$
a_{k}d_k-b_{k}c_k=2^{k-1}((ad)^k-(bc)^k)>0\quad \mbox{\rm since}\quad ad-bc>0.
$$
As above, this shows that the boundedness properties of Riccati type 
inequalities of the form $r_n\leq \phi(r_{n-1})^k$  can be deduced from the solution of the recursion $s_n=\phi_k(s_{n-1})$. Indeed, since $\phi_k$ is increasing we have
\begin{gather}\label{ricc-eq-power-k-bound}
r_n\leq \phi(r_{n-1})^k\leq  \phi_{k}^n(r_0) \leq x\vee s^{k}\quad \mbox{\rm with the fixed point}\quad
 s^{k}:= \phi_{k}(s^{k})>0.
\end{gather}
\subsection{Stability of Kalman filters}\label{theo-ricc-Kalman-intro-proof}
This section is mainly concerned with the proof of Theorem~\ref{theo-ricc-Kalman-intro}. Note that Lemma~\ref{lem-ricc} yields the existence of the positive fixed point $P_{\infty}=\phi(P_{\infty})>0$, as well as the estimate \eqref{ricc-lower-cv-2}.
The fixed point $P_\infty$ can be written in closed form in terms of the parameters $(a,b,c)$ (defined in (\ref{def-Ricc-parameters})) by the formula
\begin{equation}\label{ricc-eq-kalman-parameters-2}
P_{\infty}= \frac{(a-1)+r}{2c}\quad \mbox{\rm with}\quad r:=\sqrt{\left( a-1\right)^2+4bc} \,<a+1.
\end{equation}

In order to prove \eqref{p1-p2-q-2} we need to consider the random products defined for any $0\leq k\leq n$ and $p\geq 0$ by
 \begin{equation}\label{def-Ea}
\displaystyle  E_{k,n}(p):=\prod_{k< l\leq n}\frac{A}{1+S\phi^l(p)},
\end{equation}
since in this notation, for any given $p\geq 0$ and for any $(x_1,x_2)\in \RR^2$ and $n\geq 0$ we have
\begin{eqnarray}\label{X-p-q1-q2}
\widehat{X}_n(x_1,p)- \widehat{X}_n(x_2,p)
&=&E_{0,n}(p)~(x_1-x_2).
\end{eqnarray}

When $k=0$, we simplify notation and we write $E_{n}(p)$ instead of $E_{0,n}(p)$.
When $k=n$ we use the convention $E_{n,n}(p)=1$.
With this in mind, for any $k<l<n$ we readily check the semigroup properties
$$
E_{k,n}(p)=E_{n-k}\left(\phi^k(p)\right)\quad \mbox{\rm and}\quad
E_{n-k}(p)=E_{l-k}(p)~\times~E_{n-l}(\phi^{l-k}(p)).
$$

To obtain appropriate bounds for the products $E_{k, n}$, recall that thanks to \eqref{ul-bounds}, for any $n\geq 1$ and any starting point $P_0$ we have the uniform estimates
\begin{gather}\label{ricc-lower-cv}
R\leq \phi^n(P_0)\leq   A^2/S+R,
\end{gather}
which implies that
\begin{equation}\label{def-Ea-2}
\displaystyle E_{k, n}(p) \leq \left(\frac{\vert A\vert}{1+SB^2}\right)^{n-k}.
\end{equation}

Also observe that 
$$
\frac{A^2-1}{2}\geq \vert A\vert-1\Longrightarrow P_{\infty}=\frac{1}{S}~\left(\frac{A^2-1}{2}\right)+\frac{RS+r}{2S}>\frac{\vert A\vert-1}{S}.
$$
This implies that
 \begin{equation}\label{epsilon-infty}
 \vert A\vert (1-G_\infty C)=:1- \epsilon_{\infty}<1,
\end{equation}
where $G_\infty$ is he gain associated with the steady state $P_{\infty}$ of the Riccati equation and is given by
$$
G_{\infty}=CP_{\infty}/(C^2P_{\infty}+D^2)\Longleftrightarrow 1-G_{\infty}C={1}/{(1+SP_{\infty})}.
$$

On the other hand, using (\ref{ricc-lower-cv-2}),  for any $p\geq 0$ there exists some integer  $k(p)\geq 1$ such that for any $k> k(p)$,
$$
\frac{1+SP_{\infty}}{1+S\phi^k(p)}=1+\frac{S}{1+S\phi^k(p)}~(P_{\infty}-\phi^k(p))\leq 1+S~\kappa_1~\left(1-\epsilon_1\right)^k~\vert p-P_{\infty}\vert\leq 1+\epsilon_{\infty}.
$$
with the parameter $\epsilon_{\infty}$ defined in (\ref{epsilon-infty}).
This yields the following technical lemma.
\begin{lem}
For any $p\geq 0$ there exists some integer $k(p)\geq 1$ such that for any $n> k(p)$ we have
$$
\sqrt{\vert (\partial \phi)(\phi^n(p))\vert}=
\frac{\vert A\vert }{1+S\phi^{n}(p)} \leq 1-\epsilon^2_{\infty},
$$
with the parameter $\epsilon_{\infty}$ defined in (\ref{epsilon-infty}). In addition, 
 for any $n\geq k\geq  k(p)\geq l$ we have
\begin{equation}\label{est-expo-Ea}
\vert E_{l,n}(p)\vert\leq \left(\frac{\vert A\vert}{1+SB^2}\right)^{k(p)-l}~(1-\epsilon^2_{\infty})^{n-k(p)}
\quad \mbox{\rm and}\quad \vert E_{k,n}(p)\vert\leq (1-\epsilon^2_{\infty})^{n-k}.
\end{equation}
\end{lem}

\bigskip

We are now in position to state and to prove the main result of this section, which also completes the proof of Theorem~\ref{theo-ricc-Kalman-intro}.
\begin{theo}
For any $p_1\geq 0$ there exists some integer $k(p_1)\geq 1$
for any $n> k(p_1)$ and any $x_1,x_2\in\RR$ we have the almost sure estimate
\begin{equation}\label{p-q1-q2}
\left\vert \widehat{X}_n(x_1,p_1)- \widehat{X}_n(x_2,p_1)\right\vert\leq \left(\frac{\vert A\vert}{1+SB^2}\right)^{k(p_1)}~(1-\epsilon^2_{\infty})^{n-k(p_1)}~\vert x_1-x_2\vert.
\end{equation}
with the parameter $\epsilon_{\infty}$ defined in (\ref{epsilon-infty}).
In addition, for any $p_2\geq 0$ we have
\begin{equation}\label{p1-p2-q}
\EE\left(\left(\widehat{X}_n(x_1,p_1)- \widehat{X}_n(x_2,p_2)\right)^2\right)^{1/2}\leq~\kappa(p_1,p_2,x_2)~(1-\epsilon^2_{\infty})^{n-k(p_1)}~\vert x_1-x_2\vert.
\end{equation}
\end{theo}

\proof
The estimate (\ref{p-q1-q2}) is a direct consequence of  the formula
(\ref{X-p-q1-q2}) and the exponential estimate (\ref{est-expo-Ea}).
On the other hand, we have
\begin{eqnarray*}
  \widehat{X}_n(x,p)-X_n&=&\frac{1}{1+S\phi^n(p)} ~(A\widehat{X}_{n-1}(x,p)-X_n)+\frac{S\phi^n(p)}{1+S\phi^n(p)}~(C^{-1}~Y_n-X_n)\\
  &=&\frac{A}{1+S\phi^n(p)} ~(\widehat{X}_{n-1}(x,p)-X_{n-1})+\widehat{U}_n(p),
  \end{eqnarray*}
where
$$
\widehat{U}_n(p):=\frac{S\phi^n(p)}{1+S\phi^n(p)}~\frac{V_n}{\sqrt{S}}-\frac{1}{1+S\phi^n(p)}~BW_n\Longrightarrow \EE(\widehat{U}_n(p)^2)\leq R+1/S.
$$
Iterating, this implies that
$$
  \widehat{X}_n(x,p)-X_n= E_{n}(p) (x-X_0)+\sum_{1\leq k\leq n}~ E_{k,n}(p)~\widehat{U}_k(p)
$$
and hence
$$
\begin{array}{l}
\displaystyle \EE\left(\left(\widehat{X}_n(x,p)-X_n\right)^2  \right)=\vert E_{0,n}(p)\vert ~\EE((x-X_0)^2)+\sum_{1\leq l\leq n}~\vert E_{l,n}(p)\vert~ \EE(\widehat{U}_l(p)^2).
\end{array}
$$
Using (\ref{est-expo-Ea}) we check the uniform estimate
\begin{equation}\label{X-p-q-X}
 \EE\left(\left(\widehat{X}_n(x,p)-X_n\right)^2  \right)
 \leq \kappa_1(p)~(1+\EE((x-X_0)^2)).
\end{equation}
In the same vein, for any given $q\in \RR$ and for any $(p_1,p_2)\in\RR_+^2$ we check that
\begin{eqnarray*}
\widehat{X}_n(x,p_1)- \widehat{X}_n(x,p_2)&=&\sum_{1\leq l\leq n}~ E_{l,n}(p_1)~\rho_l(p_1,p_2)~U_l(x,p_2)
\end{eqnarray*}
with the function
$$
\begin{array}{l}
 \displaystyle
\rho_n(p_1,p_2):=\frac{S(\phi^n(p_1)-\phi^n(p_2))}{(1+S\phi^n(p_1))(1+S\phi^n(p_2))} \leq S\kappa_1^2~\left(1-\epsilon_1\right)^n~\vert p_1-p_2\vert
\quad \mbox{\rm (by (\ref{ricc-lower-cv-2})})
\end{array}$$
and the collection of random variables
$$
\begin{array}{l}
 \displaystyle
U_n(x,p):= A(X_{n-1}-\widehat{X}_{n-1}(x,p))+BW_n+V_n/\sqrt{S}\\
\\
\Longrightarrow  \EE\left(U_n(x,p)^2  \right)
 \leq A^2\kappa_1(p)~(1+\EE((x-X_0)^2))+R+1/S\quad \mbox{\rm (by (\ref{X-p-q-X}))}.
\end{array}
$$
This implies that
$$
\EE\left(\left(\widehat{X}_n(x,p_1)- \widehat{X}_n(x,p_2)\right)^2\right)^{1/2}\leq\left(\sum_{1\leq l\leq n}~\vert  E_{l,n}(p_1)\vert~\rho_l(p_1,p_2)~\right)~\kappa_2(x,p_2)
$$
with
$$
\kappa_2(x,p):=(A^2\kappa_1(p)+B^2+1/S)~(1+\EE((x-X_0)^2)).
$$
Using (\ref{est-expo-Ea}), for any $n>k(p_1)$ we have
$$
\sum_{1\leq l\leq n}~\vert  E_{l,n}(p_1)\vert~\rho_l(p_1,p_2)\leq
\kappa_3(p_1)~(1-\epsilon^2_{\infty})^{n-k(p_1)}~\vert p_1-p_2\vert.
$$
 This ends the proof of (\ref{p1-p2-q}) and hence the theorem. \cqfd

Summing the two estimates (\ref{p-q1-q2}) and (\ref{p1-p2-q}) we obtain the estimate
(\ref{p1-p2-q-2}) with $\epsilon_2=\epsilon^2_{\infty}$.

\section{Stochastic perturbation analysis}
This section is mainly concerned with the proof of Theorem~\ref{stoch-perturbation-theo}. In addition, we also consider the bias and some time-uniform bounds for the sample covariances $p_n$ and $\widehat p_n$.

\subsection{A local perturbation theorem}\label{proof-theo-stoch-intro}
We fix some parameter $N\geq 1$ and we let $\un$ be the $(N+1)$ column vector with unit entries, $I$ the $(N+1)\times(N+1)$ identity matrix and $J=\un\un^{\prime}$ the $(N+1)\times(N+1)$ matrix with unit entries, where $(\cdot)'$ denotes the transpose operator. We also let $\epsilon$ be  the $(N+1)\times(N+1)$ matrix given by
$$
\epsilon= I-{J}/{(N+1)}
\quad \Longrightarrow \quad
\epsilon^2=\epsilon.
$$
For any given $(N+1)$-column  $p$ we have
$$
\overline{p}:=\frac{1}{N+1}\sum_{1\leq j\leq N+1} p_j\quad\Longrightarrow\quad
p-\overline{p}\,\un:=\left(
\begin{array}{c}
p_1-\overline{p}\\
\vdots\\
p_{N+1}-\overline{p}
\end{array}
\right)=\epsilon p.
$$
From the mean-field evolution equations \eqref{kalman-EnKF-def} and the definitions of $m_n$ and $\widehat m_n$ in \eqref{sm-kalman-EnKF}, it is straightforward to show that
$$
\left\{
\begin{array}{rcl}
\displaystyle
(\xi^i_n-m_n)&=& A~(\widehat{\xi}^i_{n-1}-\widehat{m}_{n-1})+BW^{i,\epsilon}_n\\
\\
(\widehat{\xi}^i_n-\widehat{m}_n)&=&(1-g_nC)(\xi^i_n-m_n)-g_nDV^{i,\epsilon}_n,
\end{array}
\right.
$$
with the random variables
\begin{equation}\label{def-Weps}
W^{i,\epsilon}_n:=\sum_{1\leq j\leq N+1}\epsilon^i_j~W^j_n\quad \mbox{\rm and}\quad
V^{i,\epsilon}_n:=\sum_{1\leq j\leq N+1}\epsilon^i_j~V^j_n.
\end{equation}
Taking the square we obtain the formulae
$$
\left\{
\begin{array}{rcl}
\displaystyle
(\xi^i_n-m_n)^2&=& A^2~(\widehat{\xi}^i_{n-1}-\widehat{m}_{n-1})^2+B^2(W^{i,\epsilon}_n)^2+2AB~(\widehat{\xi}^i_{n-1}-\widehat{m}_{n-1})W^{i,\epsilon}_n\\
\\
(\widehat{\xi}^i_n-\widehat{m}_n)^2&=&(1-g_nC)^2(\xi^i_n-m_n)^2+g_n^2D^2(V^{i,\epsilon}_n)^2-2g_n(1-g_nC)D(\xi^i_n-m_n)V^{i,\epsilon}_n.
\end{array}
\right.
$$
This implies that
$$
\left\{
\begin{array}{rcl}
\displaystyle
p_n&=&\displaystyle  A^2~\widehat{p}_{n-1}+B^2+\frac{1}{\sqrt{N}}~\nu_n\\
\\
\widehat{p}_n&=&\displaystyle(1-g_nC)^2p_n+g_n^2D^2+\frac{1}{\sqrt{N}}~ \widehat{\nu}_n=(1-g_nC)p_n+\frac{1}{\sqrt{N}}~ \widehat{\nu}_n
\end{array}
\right.
$$
with the orthogonal variables
$$
\nu_n:=2AB~\chi_n+B^2~\varsigma_n\quad \mbox{\rm and}\quad
 \widehat{\nu}_n:=-2g_n(1-g_nC)D~\widehat{\chi}_n+g_n^2D^2~\widehat{\varsigma}_n
$$
and the centered random variables
\[
\begin{array}{rclcrcl}
\chi_n&:=&\displaystyle\frac{1}{\sqrt{N}}~\sum_{1\leq i\leq N+1}~(\widehat{\xi}^i_{n-1}-\widehat{m}_{n-1})W^{i,\epsilon}_n&&\varsigma_n&:=&\displaystyle\sqrt{N}~\left(\frac{1}{N}\sum_{1\leq i\leq N+1}(W^{i,\epsilon}_n)^2-1\right)\\
&&&&&&\\
\widehat{\chi}_n&:=&\displaystyle\frac{1}{\sqrt{N}}~\sum_{1\leq i\leq N+1}(\xi^i_n-m_n)V^{i,\epsilon}_n&&\widehat{\varsigma}_n&:=&\displaystyle\sqrt{N}\left(\frac{1}{N}\sum_{1\leq i\leq N+1}(V^{i,\epsilon}_n)^2-1\right).
\end{array}
\]
Further, denote by $\HH$ the $((N+1)\times(N+1))$-Helmert matrix whose first row is defined by
$$
\HH_{1,j}=1/\sqrt{N+1}\quad \forall 1\leq j\leq (N+1)\quad \Longrightarrow\quad \HH_{1,\point}=\un^{\prime}/\sqrt{N+1}.
$$
and whose $i$-th row for $2\leq i\leq (N+1)$ is given by
$$
\begin{array}{l}
\HH_{i,j}:=\overline{\HH}_{i,j}:=\left\{
\begin{array}{ccl}
1/\sqrt{i(i-1)}&\mbox{\rm if}& 1\leq j<i\\
&&\\
-(i-1)/\sqrt{i(i-1)}&\mbox{\rm if}& j=i\\
&&\\
0&\mbox{\rm if}& i<j\leq (N+1).
\end{array}\right.
\end{array}
$$
This yields the matrix decomposition
$$
\HH=\left(\begin{array}{c}\un^{\prime}/\sqrt{N+1}\\
\overline{\HH}\end{array}\right)\quad \mbox{\rm with}\quad \overline{\HH}=(\overline{\HH}_{i,j})_{2\leq i\leq (N+1), 1\leq j\leq (N+1)}\in \RR^{N\times (N+1)}.
$$
We check that $\HH$ is an orthogonal matrix and we have the square root formula
$$
\HH\HH^{\prime}=I=\HH^{\prime}\HH=\frac{1}{N+1}~J+\overline{\HH}^{\prime}~\overline{\HH} \Longrightarrow \overline{\HH}^{\prime}~\overline{\HH}=I-\frac{1}{N+1}~J=\epsilon.
$$
Also note that for any $(N+1)$-column vectors $x,w$ we have
$$
\begin{array}{l}
\displaystyle(\overline{\HH} x)^{\prime}
(\overline{\HH}x)=(x-\overline{x}\,\un)^{\prime}(w-\overline{w}\,\un)=x^{\prime}\epsilon \,w.
\end{array}
$$
In this notation, we readily check that
\begin{eqnarray*}
\displaystyle \widehat{m}_{n}&=&\displaystyle m_n+g_n~(Y_n-Cm_n)-\frac{1}{\sqrt{N+1}}~g_nD~\HH_{1,\point}\Va_n^{(N)}\\
 m_{n+1}&=&\displaystyle A~\widehat{m}_{n}+\frac{1}{\sqrt{N+1}}~B~\HH_{1,\point}\Wa_n^{(N)}
\end{eqnarray*}
with the random vectors
$$
\Wa_n^{(N)}:=\left(
\begin{array}{c}
W^1_n\\
\vdots\\
W^{N+1}_n\\
\end{array}
\right)\quad \mbox{\rm and}\quad \Va_n^{(N)}:=\left(
\begin{array}{c}
V^1_n\\
\vdots\\
V^{N+1}_n\\
\end{array}
\right).
$$
Using this notation, we may write
\[
\begin{array}{rclcrcl}
\chi_n&=&\displaystyle\sqrt{\frac{\widehat{p}_{n-1}}{N}}~\widehat{\zeta}_{n-1}^{\prime}~ \overline{\HH}~\Wa_n^{(N)}&&\varsigma_n&=&\displaystyle\frac{1}{\sqrt{N}}~\left(\left(\overline{\HH}\,\Wa^{(N)}_n\right)^{\prime}\left(\overline{\HH}\,\Wa^{(N)}_n\right)-N\right)\\
&&&&&&\\
\widehat{\chi}_n&=&\displaystyle\sqrt{\frac{p_n}{N}}~\zeta_n^{\prime}~\overline{\HH}~ \Va_n^{(N)}&&\widehat{\varsigma}_n&=&\displaystyle\frac{1}{\sqrt{N}}~\left(\left(\overline{\HH}\,\Va^{(N)}_n\right)^{\prime}\left(\overline{\HH}\,\Va^{(N)}_n\right)-N\right),
\end{array}
\]
with the random vectors on the unit sphere defined by
$$
\widehat{\zeta}_{n-1}:=\frac{\overline{\HH}\,\widehat{\xi}_{n-1}}{\sqrt{\Vert \overline{\HH}\,\widehat{\xi}_{n-1}\Vert}}\quad \mbox{\rm and}\quad
\zeta_n:=\frac{\overline{\HH}\,\xi_n}{\sqrt{\Vert \overline{\HH}\,\xi_n\Vert}}.
$$
Finally observe that
 \begin{equation}\label{ref-indp}
 \HH\,  \Va_n^{(N)}=\left(\begin{array}{c}
\HH_{1,\point}\Va_n^{(N)}
 \\
 \overline{\HH}\,  \Va_n^{(N)}
 \end{array}\right)\quad \mbox{\rm and}\quad \HH\,  \Wa_n^{(N)}=\left(\begin{array}{c}
\HH_{1,\point}\Wa_n^{(N)}
 \\
 \overline{\HH}\,  \Wa_n^{(N)}
 \end{array}\right)\sim \Na\left(0,I_{(N+1)\times(N+1)}\right)
 \end{equation}
We end the proof of Theorem~\ref{stoch-perturbation-theo} by replacing $ (\Va_n^{(N)},\Wa_n^{(N)})$ by $\left( \overline{\HH}\,  \Va_n^{(N)}, \overline{\HH}\,  \Wa_n^{(N)}\right)$.   This ends the proof of the theorem.
\cqfd

  \medskip

 \begin{cor}
For any $n\geq 0$ we have the biasedness properties
\begin{equation}\label{u-bias}
\EE(\widehat{p}_n)< \widehat{P}_n
\quad \mbox{and}\quad \EE(p_{n+1})< P_{n+1}.
\end{equation}
 \end{cor}
 \proof
 Since the function
$$
p\in \RR_+\mapsto \frac{p}{1+Sp}\in \RR_+
$$
is strictly concave, we have
$$
 \EE((1-g_nC)~p_n)< \frac{\overline{p}_n}{1+S\overline{p}_n}=(1-\overline{g}_nC)~\overline{p}_n
$$
with
$$
\overline{g}_n=C~\overline{p}_n/(C^2~\overline{p}_n
+D^2)\quad \mbox{\rm and}\quad \overline{p}_n:=\EE(p_n).
$$
This implies that
$$
\EE(\widehat{p}_n)< \displaystyle(1-\overline{g}_nC)~\overline{p}_n
$$
from which we check the Riccati differential inequality
$$
\begin{array}[c]{rcl}
\overline{p}_n&=&\displaystyle  A^2~\EE(\widehat{p}_n)+R< 
A^2~ \frac{\overline{p}_{n-1}}{1+S\overline{p}_{n-1}}+R=\phi(\overline{p}_{n-1}).
\end{array}
$$
The proof of (\ref{u-bias}) then follows from comparison properties of Riccati inequalities.
\cqfd

\subsection{Uniform raw moments estimates}
  For any $k\geq 1$ there exists some finite constant $\tau_k$ such that for any $N\geq 1$ and $n\geq 0$, the Burkholder-Davis-Gundy inequality yields the following uniform estimates
\begin{equation}\label{est-Lk-Delta}
\EE\left(\Vert \widehat{\Delta}_{n}\Vert^{k}\right)\vee
\EE\left(\Vert \Delta_{n}\Vert^{k}\right)\leq \tau_k.
\end{equation}
\textcolor{black}{\begin{prop}
  For any $k\geq 1$, any $N\geq 1$ and time horizon $n\geq 0$, we have the time-uniform estimates
  \begin{equation}\label{est-p-delta}
  \begin{array}{l}
  \EE(\widehat{p}_{n}^{k})\vee \EE\left(\widehat{\nu}_n^{k}\right)\leq \iota_k\\
  \\
\EE(p_{n}^{k})\leq \iota_k~(P_0\vee 1)^k
\quad\mbox{and}\quad \EE\left(\vert\nu_{n+1}\vert^{k}\right)\vee
\EE(\delta_{n+1}^{k})\leq \iota_k~(P_0\vee 1)^{k/2}.
\end{array}
  \end{equation}
\end{prop}
\proof 
We first prove the last assertion. To this end, for any $k\in \RR$ and $n\geq 1$, set
$
q^{k}_n:=\EE(p_{n}^{k})
$
and observe that for $k \ge 1$,
$$
2^{k-1}\phi(x)^k=2^{k-1}\frac{((A^2+RS)x+R)^k}{(Sx+1)^k}\leq \phi_{k}(x^k) \quad \mbox{\rm where}\quad
 \phi_{k}(x):=\frac{a_kx+b_k}{c_kx+1},
$$
with parameters
$$
\begin{array}{l}
(a_k,b_k,c_k)=(2^{2(k-1)}(A^2+RS)^k,2^{2(k-1)}R^{k},S^k)\in\RR_+^3\\
\\
\Longrightarrow\quad a_k-b_kc_k=2^{2(k-1)}\left((A^2+RS)^k-R^{k}S^k\right)>0.
\end{array}
$$
Combining the above with Corollary \ref{cor-srre} and the inequality $|a + b|^k \le 2^{k-1}(|a|^k + |b|^k)$, $a, b \in \RR$, $k \ge 1$, we obtain
$$
p_{n+1}^{2k}\leq \phi_{2k}(p_n^{2k})+\frac{2^{2k-1}}{N^k}~\left(
A^{4k}~\widehat{\nu}_n^{2k}+\nu^{2k}_{n+1}\right).
$$
By Jensen's inequality, we check that
$$
q^{2k}_{n+1}\leq \phi_{2k}(q^{2k}_n)+\frac{2^{2k-1}}{N^k}~\left(
A^{4k}~\EE(\widehat{\nu}_n^{2k})+\EE(\nu^{2k}_{n+1})\right).
$$
{\color{black}To control the expectations on the right hand side above, note that for any $k\geq 1$ we have the crude estimates
\begin{eqnarray*}
(g_nD)^{2k}&=&\left(\frac{Sp_n}{1+Sp_n}~\frac{p_n}{1+Sp_n}\right)^{k}\leq  S^{-k},\\
(g_nD(1-g_nC))^{k}~p_{n}^{k/2}
&=&\left(\frac{Sp_n}{1+Sp_n}~\frac{1}{1+Sp_n}~\left(\frac{p_n}{1+Sp_n}\right)^2\right)^{k/2}\leq S^{-k}.
\end{eqnarray*}
Combining these estimates with (\ref{est-Lk-Delta}) in \eqref{def-loc-fluctuation}, we check that
\begin{equation}\label{moments-w-nu-k}
\EE\left(\widehat{\nu}_n^{k}\right)\leq \iota_1(k)
\quad
\mbox{\rm and}
\quad
 \EE(\nu_{n+1}^{k})\leq \iota_2(k)~\left(1\vee q_n^{k/2}\right) = \iota_2(k)~\left(1\vee q_n\right)^{k/2}.
\end{equation}
Now, \eqref{ricc-lower-cv} and \eqref{u-bias} imply that
\begin{equation}\label{qn}
q_n\leq \iota_0~(P_0\vee 1).
\end{equation}
Combining this with \eqref{moments-w-nu-k} yields the required bounds for $\nu_n^k$, $\widehat \nu_n^k$ and $\delta_n^k$. Returning to the moments of $p_n$, note that the previous estimates imply that
$$
q^{2k}_{n+1}\leq \phi_{2k}(q^{2k}_n)+\frac{\iota_5(k)}{N^k}~(q^k_n\vee 1).
$$
Applying the above estimate to $k=1$, along with \eqref{qn}, we have
$$
q^{2}_{n+1}\leq \phi_{2}(q^{2}_n)+\frac{\iota_1}{N}~( P_0\vee 1).
$$
Using the comparison properties \eqref{ricc-eq-epsilon-bound} and \eqref{ricc-eq-power-k-bound} we conclude that
$$
\sup_{N\geq 1}\sup_{n\geq 0}q^{2}_{n}<\iota_2~(P_0\vee 1)^2.
$$
Iterating the argument for any $k\geq 1$ we readily check that
$$
\sup_{N\geq 1}\sup_{n\geq 0}q^{k}_{n}<\iota_k~(P_0\vee 1)^k
$$
and therefore
$$
\EE\left(\vert\nu_{n+1}\vert^{k}\right)\leq 
 \iota_6(k)~\left(1\vee q_n\right)^{k/2}\leq \iota_7(k)~(P_0\vee 1)^{k/2}
\quad \mbox{\rm and}\quad
\EE\left(\vert\delta_{n+1}\vert^k\right)\leq\iota_k~(P_0\vee 1)^{k/2}
$$}
Finally, we have
$$
\widehat{p}_n=\frac{p_n}{1+Sp_n}+\frac{1}{\sqrt{N}}~\widehat{\nu}_n\Longrightarrow
 \EE(\widehat{p}_{n}^{k})\leq S^{-k} +\frac{1}{N^{k/2}}~ \EE(\vert \widehat{\nu}_n\vert^{k})\leq \iota_k,
$$
which ends the proof of the proposition.}
\cqfd

\section{Stochastic Riccati  equations}\label{stoch-perturbation-sec}
In this section we focus on the proof of Theorem \ref{theo-u-p}, as well as the central limit theorems presented in section \ref{sec-clt-intro}. 
\subsection{Perturbation analysis}\label{mean-error-P-sec}
We start by proving the time uniform estimates \eqref{time-unif-bound}.
To this end, first note that we have the telescoping sum formula
\begin{equation}\label{telescoping}
p_n-P_n=\left(\phi^n(p_0)-\phi^n(P_0)\right)+\sum_{1\leq k\leq n} \left(
\phi^{n-k}(p_k)-\phi^{n-(k-1)}(p_{k-1})\right).
\end{equation}
To deal with the summands above, observe that
$$
\begin{array}{rcl}
\phi^{n-k}(p_k)-\phi^{n-(k-1)}(p_{k-1})
&=&\phi^{n-k}(p_k)-\phi^{n-k}(\phi(p_{k-1}))\\
&&\\
&=&\displaystyle \phi^{n-k}\left(\phi(p_{k-1})+\frac{1}{\sqrt{N}}~ \delta_k\right)-\phi^{n-k}\left(\phi(p_{k-1})\right).
\end{array}$$
Using the Lipschitz estimates (\ref{expo-bounds}) we check that
$$
\vert 
\phi^{n-k}(p_k)-\phi^{n-k}(\phi(p_{k-1}))
\vert\leq \frac{\kappa_1}{\sqrt{N}}~(1-\epsilon_1)^{n-k}
~\left\vert  \delta_k\right\vert,
$$
with the parameter $\epsilon_1$ defined in \eqref{ricc-lower-cv-2}.
Similarly, we have
$$
\vert \phi^n(p_0)-\phi^n(P_0)\vert \leq \frac{\kappa_2}{\sqrt{N}}~(1-\epsilon_1)^{n}
~\left\vert  \nu_0\right\vert.
$$
This yields the almost sure estimate
$$
\begin{array}{l}
 \displaystyle  \sqrt{N}~\vert p_n-P_n\vert  \leq\kappa_3~\left[(1-\epsilon_1)^{n}\vert \nu_0\vert+\sum_{1\leq k\leq n}  ~(1-\epsilon_1)^{n-k}
~\left\vert \delta_k\right\vert\right].
\end{array}
$$
Next we note that
$$
\vert\widehat{p}_n-\widehat{P}_n\vert=\left\vert\frac{p_n-P_n}{(1+SP_n)(1+Sp_n)}\right\vert+\frac{1}{\sqrt{N}}~\vert\widehat{\nu}_n\vert\leq 
\left\vert p_n-P_n\right\vert+\frac{1}{\sqrt{N}}~\vert\widehat{\nu}_n\vert.
$$
Finally, we note the following decomposition for the gain parameters
\begin{align*}
\sqrt{N}(g_n - G_n) 
&= \sqrt{N}C\left(\frac{p_n}{C^2p_n + D^2} - \frac{P_n}{C^2P_n + D^2}\right) \\
&= CD^2 \frac{\sqrt{N}(p_n - P_n)}{(C^2p_n + D^2)(C^2P_n + D^2)}.
\end{align*}
The above estimates and decompositions allow one to derive several mean error bounds as soon as we control the moments of the local errors. For instance, using the uniform moments estimates (\ref{est-p-delta}) we obtain the mean error estimates stated in Theorem~\ref{theo-u-p}. \cqfd

\subsection{Fluctuation analysis}\label{theo-loc-tcl-proof}

This section is mainly concerned with the proofs of Theorem~\ref{theo-loc-tcl} and Theorem \ref{theo-clt-p-intro}. We start with some technical results that will immediately lead to the proof of the former.


\begin{lem}\label{lem-Helmert}
The characteristic function of random vector $\Delta$ defined in (\ref{def-Delta}) satisfies
for any $w\in \RR^3$ we have the Gaussian type  estimate
\begin{equation}\label{charact-Z}
\left\vert \EE\left(e^{i w^{\prime}\Delta }\right)-e^{-\frac{\Vert w\Vert ^2}{2}}\right\vert\leq \epsilon(w)~e^{-\frac{\Vert w\Vert ^2}{2}},
\end{equation}
with the function
\begin{equation}\label{charact-Z-epsilon}
\epsilon(w):=\vert w_3\vert\sqrt{\frac{2}{N}}\left( \frac{w_2^2}{2}+w_3^2
\right)~\exp{\left(\vert w_3\vert\sqrt{\frac{2}{N}}\left( \frac{w_2^2}{2}+w_3^2
\right)\right)}.
\end{equation}
\end{lem}

\medskip

The proof of the above lemma follows elementary calculations. For the convenience of the reader a detailed proof is provided in the appendix.

\medskip

\begin{lem}\label{lem-tex-delta}
For any $n\geq 0$ we have the weak convergence
$$
\left(\Delta_0,(\widehat{\Delta}^i_k,\Delta^i_{k+1})_{0\leq k\leq n,~1\leq i\leq 3}\right)\hookrightarrow_{N\rightarrow\infty}~\left((Z^i_0)_{1\leq i\leq 2},(\widehat{Z}^i_k,Z^i_{k+1})_{0\leq k\leq n,~1\leq i\leq 3}\right),
$$
where the $Z^i_k,\widehat{Z}^i_k$ are independent sequences of independent and centered Gaussian random variables with unit variance.
\end{lem}

\proof
Applying (\ref{charact-Z}) for any $w\in\RR^3$ we check  the rather crude estimates
$$
\left\vert \EE\left(e^{i~w^{\prime}\Delta_{n+1}}\right)-e^{-\frac{\Vert w\Vert ^2}{2}}\right\vert\leq \epsilon(w)~
\quad \mbox{\rm and}\quad
\left\vert \EE\left(e^{i~w^{\prime}\widehat{\Delta}_{n}}\right)-e^{-\frac{\Vert w\Vert ^2}{2}}\right\vert\leq \epsilon(w)~
$$
with the function $\epsilon(w)\longrightarrow_{N\rightarrow \infty}0$ defined in (\ref{charact-Z-epsilon}). 
For any $u,v\in\RR^3$ and $n\geq 0$ we have the decomposition
$$
\begin{array}{l}
\displaystyle\EE\left(e^{i~u^{\prime}\Delta_{n+1}}~e^{i~v^{\prime}\widehat{\Delta}_{n}}\right)-e^{-\frac{\Vert u\Vert ^2}{2}}e^{-\frac{\Vert v\Vert ^2}{2}}\\
\\
\displaystyle=\EE\left(e^{i~v^{\prime}\widehat{\Delta}_{n}}\right)~\left(\EE\left(e^{i~u^{\prime}\Delta_{n+1}}\right)-e^{-\frac{\Vert u\Vert ^2}{2}}\right)+\left(\EE\left(
e^{i~v^{\prime}\widehat{\Delta}_{n}}\right)-e^{-\frac{\Vert v\Vert^2}{2}}\right)e^{-\frac{\Vert u\Vert ^2}{2}}.
\end{array}$$
This yields the estimate
$$
\left\vert
\EE\left(e^{i~u^{\prime}\Delta_{n+1}}~e^{i~v^{\prime}\widehat{\Delta}_{n}}\right)-e^{-\frac{\Vert u\Vert ^2+\Vert v\Vert^2}{2}}\right\vert\leq
\epsilon(u,v):=\epsilon(u)+\epsilon(v).
$$
Consider the sequence of  random vectors
$$
\Lambda_n:=\left(
\begin{array}{c}
\widehat{\Delta}_{n}\\
\Delta_{n+1}
\end{array}\right)\in \RR^6.
$$
In this notation, for any $n\geq 0$ and $w_n\in\RR^6$ we have proven the following estimate
$$
\left\vert
\EE\left(e^{i~w_n^{\prime}\Lambda_n}\right)-e^{-\frac{\Vert w_n\Vert ^2}{2}}\right\vert\leq 
\epsilon(w_n).
$$
Using the decomposition
$$
\begin{array}{l}
\EE\left(e^{i~\left(w_{n-1}^{\prime}\Lambda_{n-1}+w_n^{\prime}\Lambda_n\right)}\right)-e^{-\frac{\Vert w_{n-1}\Vert ^2}{2}}e^{-\frac{\Vert w_n\Vert ^2}{2}}\\
\\
=
\EE\left(e^{i~w_{n-1}^{\prime}\Lambda_{n-1}}\right)~\left(\EE\left(
e^{i~w_n^{\prime}\Lambda_n}\right)-e^{-\frac{\Vert w_n\Vert ^2}{2}}\right)
\\
\\
\hskip3cm+\left(\EE\left(e^{i~w_{n-1}^{\prime}\Lambda_{n-1}}\right)-e^{-\frac{\Vert w_{n-1}\Vert ^2}{2}}~\right)e^{-\frac{\Vert w_n\Vert ^2}{2}},
\end{array}
$$
we also check that
$$
\begin{array}{l}
\left\vert\EE\left(e^{i~\left(w_{n-1}^{\prime}\Lambda_{n-1}+w_n^{\prime}\Lambda_n\right)}\right)-e^{-\frac{\Vert w_{n-1}\Vert ^2}{2}}e^{-\frac{\Vert w_n\Vert ^2}{2}}\right\vert
\leq \epsilon(w_{n-1})+ 
\epsilon(w_n).
\end{array}
$$
Iterating the argument, we conclude that
$$
\begin{array}{l}
\left\vert
\EE\left(e^{i~\left(u_1\Delta^1_0+u_2\Delta^2_0+\sum_{0\leq k\leq n}w_{k}^{\prime}\Lambda_k\right)}\right)-
e^{-\frac{ u_1^2/2+u_2^2/2+
\sum_{0\leq k\leq n}\Vert w_{k}\Vert ^2
}{2}
}
\right\vert.\\
\\
\leq \epsilon(u_1,0,u_2)+ \sum_{0\leq k\leq n}\epsilon(w_{k})\longrightarrow_{N\rightarrow \infty}0,
\end{array}
$$
which ends the proof of the lemma.\cqfd

Combining Lemma \ref{lem-tex-delta} with Theorem~\ref{theo-u-p} and Slutsky's lemma yields Theorem~\ref{theo-loc-tcl}.

We are now in position to prove theorem~\ref{theo-clt-p-intro} and the bias estimates (\ref{bias-estimates}).

\bigskip

{\bf Proof of Theorem~\ref{theo-clt-p-intro} :}

By the second order estimate (\ref{2nd-order}) we have
$$
\begin{array}{l}
\displaystyle\vert\phi^{n-k}(p_k)-\phi^{n-(k-1)}(p_{k-1})-\partial  \phi^{n-k}\left(\phi(p_{k-1})\right)\frac{1}{\sqrt{N}}~\delta_k\vert\leq 
 \frac{\iota_1}{N}~(1-\epsilon_1)^{n-k}~\delta_k^2.
\end{array}$$
Combining this with the telescoping sum formula (\ref{telescoping}) yields the decomposition
$$
\QQ^N_n:=\sqrt{N}(p_n-P_n)=\partial\phi^n(P_0)~\nu_0+\sum_{1\leq k\leq n} \partial  \phi^{n-k}\left(\phi(p_{k-1})\right)~\delta_k+\theta_1^N(1),
$$
with the remainder term
$$
\vert\theta^N_n(1)\vert\leq\epsilon_n^N(1):= \frac{\iota_2}{\sqrt{N}}~(1-\epsilon_1)^n~\nu_0^2~+\frac{\iota_1}{\sqrt{N}}~\sum_{1\leq k\leq n} ~(1-\epsilon_1)^{n-k}~\delta_k^2\longrightarrow_{N\rightarrow\infty}~0\quad \mbox{a.s.}.
$$
On the other hand, using the Lipschitz estimates (\ref{expo-bounds}) and (\ref{Lip-2-ricc}), we have
$$
\begin{array}{l}
 \displaystyle\left\vert
\sum_{1\leq k\leq n} \left[\partial
\phi^{n-k}(\phi(p_{k-1}))-\partial
\phi^{n-k}(\phi(P_{k-1}))\right]~\delta_k~\right\vert\\
\\
 \displaystyle\leq \epsilon_n^N(2):=\iota_3 \sum_{1\leq k\leq n}~(1-\epsilon)^{n-k}~\vert\delta_k\vert~ \vert
p_{k-1}-
P_{k-1}\vert\longrightarrow_{N\rightarrow\infty}~0\quad \mbox{a.s.}.
 \end{array}$$
 This yields the formula
 $$
\QQ^N_n=F_n\left(\nu_0,\delta_1,\ldots,\delta_n\right)+\theta_n^N(2)
$$
with the function
$$
\displaystyle F_n\left(\nu_0,\delta_1,\ldots,\delta_n\right):=\partial\phi^n(P_0)~\nu_0+\sum_{1\leq k\leq n} \partial  \phi^{n-k}\left(P_k\right)~\delta_k,
$$
and the remainder term $\theta_n^N(2)$ satisfying
$$
\vert\theta_n^N(2)\vert\leq \epsilon_n^N(1)+\epsilon_n^N(2).
$$
Combining Slutsky's lemma with the continuous mapping theorem and Corollary~\ref{cor-Z},
for any time horizon $n\geq 0$ we conclude that
$$
(F_k\left(\nu_0,\delta_1,\ldots,\delta_k\right))_{0\leq k\leq n}\hooklongrightarrow_{N\rightarrow\infty}~
(F_k\left(\ZZ_0,\ZZ_1,\ldots,\ZZ_k\right))_{0\leq k\leq n}$$
with the Gaussian random variables $\ZZ_k$ defined in Corollary~\ref{cor-Z}. Thus, we have shown that
$$
\QQ_n^N \hooklongrightarrow_{N \to \infty} F_n\left(\ZZ_0,\ZZ_1,\ldots,\ZZ_n\right) =: \QQ_n.
$$
The recursive formulation (\ref{recursion-QQ}) comes from the decomposition
$$
\widehat{\QQ}^N_n=
\frac{1}{(1+SP_n)(1+Sp_n)}~\QQ^N_n+\widehat{\nu}_n\hooklongrightarrow_{N\rightarrow\infty} \widehat{\QQ}_n:=\frac{1}{(1+SP_n)^2}~\QQ_n+\widehat{\VV}_n.
$$
In the same vein, we have
$$
\QQ^N_{n+1}=A^2~\widehat{\QQ}^N_n+\nu_{n+1}\Longrightarrow \QQ_{n+1}=A^2\widehat{\QQ}_n+\VV_{n+1}.
$$ 
\cqfd
The proof of the uniform bias estimates (\ref{bias-estimates}) follows the second order Taylor expansions discussed in the proof of Theorem~\ref{theo-clt-p-intro} as we will now demonstrate.

\medskip

{\bf Proof of the bias estimate (\ref{bias-estimates}):}
Using the  second order Taylor expansions discussed in the proof of Theorem~\ref{theo-clt-p-intro} we have the bias estimates
$$
\begin{array}{l}
 \displaystyle
0\leq P_n-\EE(p_n)\leq\frac{ \iota_4}{N}~\left[\EE(\nu_0^2)~(1-\epsilon)^n+~\sum_{1\leq k\leq n}(1-\epsilon)^{n-k}~ \EE(\delta_k^2)\right]\\
 \\
  \displaystyle\hskip5cm +\frac{\iota_5}{\sqrt{N}} \sum_{1\leq k\leq n}~(1-\epsilon)^{n-k}~\EE(\vert\delta_k\vert~ \vert
p_{k-1}-
P_{k-1}\vert).
 \end{array}
$$
The first bias estimte stated in \eqref{bias-estimates} now follows from \eqref{est-p-delta} and the estimates stated in~\eqref{time-unif-bound}.
\textcolor{black}{Indeed, using \eqref{est-p-delta} we check that
\begin{eqnarray*}
0\leq P_n-\EE(p_n)&\leq& \frac{ \iota_6}{N}~\left[P_0~(1-\epsilon)^n+~(1\vee P_0)\sum_{1\leq k\leq n}(1-\epsilon)^{n-k}\right] \\
&&\hskip1.5cm+\frac{\iota_7}{\sqrt{N}}~(1\vee P_0)~ \sum_{1\leq k\leq n}~(1-\epsilon)^{n-k}~~\EE( \vert
p_{k-1}-
P_{k-1}\vert^2)^{1/2}.
\end{eqnarray*}
The end of the proof of the l.h.s. estimate in \eqref{bias-estimates} is now a consequence of \eqref{time-unif-bound}.
We also have the second order decomposition 
\begin{eqnarray*}
g_n - G_n
&=& \frac{CD^2}{(C^2P_n + D^2)^2}~(p_n - P_n)-\frac{CD^2}{(C^2P_n + D^2)^2(C^2p_n + D^2)}~(p_n - P_n)^2
\end{eqnarray*}
Using the bias estimate stated in the l.h.s. of 
 (\ref{bias-estimates}) we readily check that
$$
0\leq G_n-\EE(g_n) \leq \frac{CD^2}{(C^2P_n + D^2)^2}~\frac{\iota_1}{N}~\left[1\vee P_0\right]^{2}+\frac{C}{(C^2P_n + D^2)^2}~\EE\left((p_n - P_n)^2\right)
$$
Again, using the $\LL_k$-mean error estimates stated in~\eqref{time-unif-bound} we check the r.h.s. estimate stated in 
 (\ref{bias-estimates}). This ends the proof of the bias estimate.} \cqfd

\subsection{Inverse raw moments}
We end the section with some bounds on the moments of ratios of the form $\phi(p_n)/p_{n+1}$, $n \ge 1$. 

Using (\ref{wp-p-chi}) we have the formulae, which will be used throughout this section.
\begin{equation}\label{inv-p-eq}
\frac{1}{\widehat{p}_n}=
\left(S+\frac{1}{p_n}\right)~\widehat{\Upsilon}_n(p_n)\quad \mbox{\rm and}\quad
\frac{1}{p_{n+1}}=\frac{1}{A^2\widehat{p}_n+R}~\Upsilon_{n+1}(\widehat{p}_n),
\end{equation}
with the inverse non-central $\chi$-square random variables
$$
\widehat{\Upsilon}_n(p_n):=\frac{1+1/(Sp_n)}{\frac{1}{N}~\widehat{\chi}^{(n,2)}_{N,N/(Sp_n)}}\quad \mbox{\rm and}\quad \Upsilon_{n+1}(\widehat{p}_n):=\frac{1+\left(A^2/R\right)\widehat{p}_n}{\frac{1}{N}~\chi^{(n+1,2)}_{N,N\left(A^2/R\right)\widehat{p}_n}}.
$$

\begin{lem}
For any $n\geq 0$ we have
\begin{equation}\label{inv-wp-2}
\begin{array}{l}
\displaystyle\frac{\phi(p_n)}{p_{n+1}}-1=\frac{1}{N}\frac{A^4}{\phi(p_n)}~\frac{\widehat{\nu}_n^2}{\phi(p_n)+\frac{A^2}{\sqrt{N}}~\widehat{\nu}_n}-\frac{A^2}{\sqrt{N}}~ \frac{\widehat{\nu}_n}{\phi(p_n)}\\
\\
\hskip5cm \displaystyle+\frac{\phi(p_n)}{A^2\widehat{p}_n+R}~\left(\Upsilon_{n+1}(\widehat{p}_n)-1\right).
\end{array}
\end{equation}
\end{lem}
\proof
First recall that 
$$
\widehat{p}_n=\displaystyle(1-g_nC)~p_n+\frac{1}{\sqrt{N}}~ \widehat{\nu}_n.
$$
Combining this with the second identity in \eqref{inv-p-eq}, we have 
$$
\begin{array}{l}
 \displaystyle \frac{\phi(p_n)}{p_{n+1}}=\frac{1}{1+\frac{A^2}{\sqrt{N}}~ \frac{\widehat{\nu}_n}{\phi(p_n)}}+\frac{\phi(p_n)}{A^2\widehat{p}_n+R}~\left(\Upsilon_{n+1}(\widehat{p}_n)-1\right).
 \end{array} 
$$
The result now follows from the decomposition
$$
\frac{1}{1+\frac{A^2}{\sqrt{N}}~ \frac{\widehat{\nu}_n}{\phi(p_n)}}=1-\frac{A^2}{\sqrt{N}}~ \frac{\widehat{\nu}_n}{\phi(p_n)}+\frac{1}{N}\frac{A^4}{\phi(p_n)}~\frac{\widehat{\nu}_n^2}{\phi(p_n)+\frac{A^2}{\sqrt{N}}~\widehat{\nu}_n}.
$$
\cqfd

\begin{prop}
For any $n\geq 0$ and any even number of particles $N>4$ we have the almost sure uniform drift estimates
\begin{equation}\label{lyap-p-wp}
1\leq \EE\left(\frac{\phi(p_n)}{p_{n+1}}~|~p_{n}\right)\leq 1+\frac{\iota}{N}.
\end{equation}

\end{prop}

\proof
First note that from the previous lemma, we have
\begin{align*}
  \displaystyle   \EE\left(\frac{\phi(p_n)}{p_{n+1}}~\bigg|~p_n,\widehat{p}_n\right)
   \displaystyle =1-\frac{A^2}{\sqrt{N}}~ \frac{\widehat{\nu}_n}{\phi(p_n)}+\frac{1}{N}\frac{A^4}{\phi(p_n)}~\frac{\widehat{\nu}_n^2}{\phi(p_n)+\frac{A^2}{\sqrt{N}}~\widehat{\nu}_n}\\
+~\frac{\phi(p_n)}{A^2\widehat{p}_n+R}~\EE\left(\left(\Upsilon_{n+1}(\widehat{p}_n)-1\right)~|~\widehat{p}_n\right).
 \end{align*} 
Due to \eqref{inv-moments-k}, for any even number of particles $N>4$ we have the estimate
$$
0\leq 
\EE\left(\Upsilon_n(x)\right)-1\leq 
\frac{4}{N}~\left(1+\frac{A^2}{R}~x\right)~\left(1+\frac{4}{N-4}\right).
$$
Since both $\widehat{p}_n$ and $\phi(p_n)$ are bounded, it follows that
$$
  \EE\left(\frac{\phi(p_n)}{p_{n+1}}~\bigg|~p_n,\widehat{p}_n\right) \leq 1-\frac{A^2}{\sqrt{N}}~ \frac{\widehat{\nu}_n}{\phi(p_n)}+\frac{1}{N}\frac{A^4}{\phi(p_n)}~\frac{\widehat{\nu}_n^2}{\phi(p_n)+\frac{A^2}{\sqrt{N}}~\widehat{\nu}_n}+\frac{\iota_1}{N}.
  $$
  for some finite constant $\iota_1$.
Taking the expectation we check that
\begin{align*}
    \displaystyle \EE\left(\frac{\phi(p_n)}{p_{n+1}}~|~p_n\right)
   \displaystyle &\leq \EE\left(1+\frac{1}{N}\frac{A^4}{\phi(p_n)}~\frac{\widehat{\nu}_n^2}{\phi(p_n)+\frac{A^2}{\sqrt{N}}~\widehat{\nu}_n}~|~p_n\right)+\frac{\iota_1}{N}\\\
      \\
         \displaystyle &\leq   1+\frac{1}{N}\left(\iota_1+\frac{A^4}{R^2}~\EE\left(~\widehat{\nu}_n^2~|~p_n\right) \right),
 \end{align*} 
 where the last line follows from the estimates
\begin{equation}
\phi(p_n)\geq R\quad\mbox{\rm and}\quad R\leq 
A^2\widehat{p}_n+R=\phi(p_n)+\frac{A^2}{\sqrt{N}}~ \widehat{\nu}_n.
\label{eq:lowerbounds}
\end{equation}
The proposition now follows from the almost sure estimate
\begin{eqnarray}
\EE\left(\widehat{\nu}_n^2~|~p_n\right)&=&(2g_n(1-g_nC)D~\sqrt{p_n})^2+(\sqrt{2}~g_n^2D^2)^2 \notag\\
&=&2S
\frac{p_n^3}{(1+Sp_n)^4}~\left(2+Sp_n\right)\leq 4S~\frac{p_n^3}{(1+Sp_n)^3}\leq \frac{4}{S^2}.
\label{eq:est-nuhat}
\end{eqnarray}
\cqfd

\begin{theo}
For any even $k\geq 2$ there exists some  finite constant $\iota_k$ such that
for any $n\geq 0$ and any even parameter $N>2(k+2)$ we have the almost sure uniform estimates
\begin{equation}\label{ratio-k}
1\leq 
\EE\left(\left(\frac{\phi(p_n)}{p_{n+1}}\right)^k~|~p_n\right)\leq 1+\frac{\iota_k}{N}.
\end{equation}
\end{theo}

\proof
We will show that \eqref{ratio-k} holds for the cases $k = 2, 3$ and then outline how the general case may be obtained inductively.  
By \eqref{inv-moments-k-3}, for any $n\geq 1$ any  $k\geq 0$ and  any even $N>2(k+2)$ we have the uniform estimates
\begin{gather}\label{resum-2}
\EE\left(\left(\frac{\Upsilon_{n}(x)-1}{A^2x+R}\right)^k\right)\leq 
 \frac{\iota_k}{N}.
\end{gather}

Now, recall that
\begin{equation}\label{pre-binomial-2}
\frac{\phi(p_n)}{p_{n+1}}-1=\widehat{\alpha}_n(p_n)+\beta_{n+1}(\widehat{p}_n),
\end{equation}
with the random variables
$$
\widehat{\alpha}_n(p_n)=-\frac{1}{\sqrt{N}}~\frac{A^2}{A^2\widehat{p}_n+R}~ \widehat{\nu}_n
\quad
\mbox{\rm and}
\quad
\beta_{n+1}(\widehat{p}_n):=\phi(p_n)~\frac{\Upsilon_{n+1}(\widehat{p}_n)-1}{A^2\widehat{p}_n+R}.
$$
Using the estimates \eqref{eq:lowerbounds},  \eqref{eq:est-nuhat} and \eqref{resum-2},
we check the almost sure uniform estimate
$$
\begin{array}{l}
\displaystyle
\EE\left(\left(\widehat{\alpha}_n(p_n)+\beta_{n+1}(\widehat{p}_n)
\right)~|~p_n\right)\\
\\
\displaystyle\leq \frac{1}{N}\frac{A^4}{R^2}~\EE(\widehat{\nu}_n^2~|~p_n)+\phi(p_n)~\EE\left(~\frac{\Upsilon_{n+1}(\widehat{p}_n)-1}{A^2\widehat{p}_n+R}
~|~p_n\right)\leq \frac{\iota_1}{N}.
\end{array}
$$
In the same vein, we have
$$
\begin{array}{l}
\displaystyle
\EE\left(\left(\widehat{\alpha}_n(p_n)+\beta_{n+1}(\widehat{p}_n)
\right)^2~|~p_n,\widehat{p}_n\right)\\
\\
\displaystyle=\widehat{\alpha}_n(p_n)^2+\phi(p_n)^2~\EE\left(\left(\frac{\Upsilon_{n+1}(\widehat{p}_n)-1}{A^2\widehat{p}_n+R}
\right)^2~|~p_n,\widehat{p}_n\right)\\
\\
\displaystyle\hskip3cm+2~\phi(p_n)~\widehat{\alpha}_n(p_n)~\EE\left(\frac{\Upsilon_{n+1}(\widehat{p}_n)-1}{A^2\widehat{p}_n+R}~|~p_n,\widehat{p}_n\right).\end{array}
$$
On the other hand, we have
$$
\begin{array}{l}
\displaystyle
\widehat{\alpha}_n(p_n)~\EE\left(\frac{\Upsilon_{n+1}(\widehat{p}_n)-1}{A^2\widehat{p}_n+R}~|~p_n,\widehat{p}_n\right)\\
\\
\displaystyle\leq\widehat{\alpha}_n(p_n)^2+\EE\left(\frac{\Upsilon_{n+1}(\widehat{p}_n)-1}{A^2\widehat{p}_n+R}~|~p_n,\widehat{p}_n\right)^2\leq \widehat{\alpha}_n(p_n)^2+\iota_2/N^2.
\end{array}
$$
Along with the almost sure uniform bound
$$
\EE\left(\widehat{\alpha}_n(p_n)^2~|~p_n\right)\leq {\iota_3}/{N},
$$
which follows from \eqref{est-p-delta}, we conclude that
$$
\EE\left(\left(\widehat{\alpha}_n(p_n)+\beta_{n+1}(\widehat{p}_n)
\right)^2~|~p_n\right)\leq  {\iota_4}/{N}.
$$
Combining this with \eqref{pre-binomial-2}, it follows that
$$
1\leq \EE\left(\left(\frac{\phi(p_n)}{p_{n+1}}\right)^2~|~p_n=x\right)\leq 1+{\iota_5}/{N}.
$$

For odd powers, for instance we have
$$
\begin{array}{l}
\displaystyle
\EE\left(\left(\widehat{\alpha}_n(p_n)+\beta_{n+1}(\widehat{p}_n)\right)^3~|~p_n,\widehat{p}_n\right)\\
\\
\displaystyle=
\widehat{\alpha}_n(p_n)^3+3~\widehat{\alpha}_n(p_n)^2\phi(p_n)~\EE\left(\frac{\Upsilon_{n+1}(\widehat{p}_n)-1}{A^2\widehat{p}_n+R}~|~\widehat{p}_n\right)\\
\\
\displaystyle+3~\widehat{\alpha}_n(p_n)\phi(p_n)^2~\EE\left(\left(\frac{\Upsilon_{n+1}(\widehat{p}_n)-1}{A^2\widehat{p}_n+R}\right)^2~|~\widehat{p}_n\right)+\phi(p_n)^3~\EE\left(\left(~\frac{\Upsilon_{n+1}(\widehat{p}_n)-1}{A^2\widehat{p}_n+B^2}\right)^3~|~\widehat{p}_n\right).
\end{array}$$
Now observe that
$$
\begin{array}{l}
\displaystyle 6~\widehat{\alpha}_n(p_n)\phi(p_n)^2~\EE\left(\left(\frac{\Upsilon_{n+1}(\widehat{p}_n)-1}{A^2\widehat{p}_n+R}\right)^2~|~\widehat{p}_n\right)\\
\\
\displaystyle\leq \left(3\widehat{\alpha}_n(p_n)\phi(p_n)^2\right)^2+\EE\left(\left(\frac{\Upsilon_{n+1}(\widehat{p}_n)-1}{A^2\widehat{p}_n+R}\right)^4~|~\widehat{p}_n\right)\leq \left(3\widehat{\alpha}_n(p_n)\phi(p_n)^2\right)^2+\frac{\iota_6}{N}.
\end{array}$$
Using the estimates (\ref{resum-2}), this implies that
$$
\begin{array}{l}
\displaystyle
\EE\left(\left(\widehat{\alpha}_n(p_n)+\beta_{n+1}(\widehat{p}_n)\right)^3~|~p_n\right)\\
\\
\displaystyle\leq 
\EE\left(\widehat{\alpha}_n(p_n)^3~|~p_n\right)+\frac{1}{2}\left(3\phi(p_n)^2\right)^2~\EE\left(\widehat{\alpha}_n(p_n)^2~|~p_n\right)+\frac{\iota_6}{2N}
\\
\\
\displaystyle\hskip3cm+3\,\phi(p_n)~\EE\left(\widehat{\alpha}_n(p_n)^2~|~p_n\right)~\frac{\iota_1}{N}+\phi(p_n)^3~\frac{\iota_3}{N}.
\end{array}$$
Using the fact that
$$
\EE\left(\vert \widehat{\alpha}_n(p_n)\vert^3~|~p_n\right)\leq \frac{\iota_7}{N\sqrt{N}},
$$
it follows that
$$
\begin{array}{l}
\displaystyle
\EE\left(\left(\widehat{\alpha}_n(p_n)+\beta_{n+1}(\widehat{p}_n)\right)^3~|~p_n\right)\leq \frac{\iota_8}{N}.
\end{array}$$
Using \eqref{pre-binomial-2} and Jensen's inequality we check that
$$
1\leq \EE\left(\left(\frac{\phi(p_n)}{p_{n+1}}\right)^3~|~p_n\right)\leq 1+\frac{\iota_9}{N}.
$$
More generally, for any $l\geq 2$ and $k\geq 1$ we have
$$
\EE\left(\vert \widehat{\alpha}_n(p_n)\vert^l~|~p_n=x\right)\leq \frac{\iota_l(1)}{N}\quad \mbox{\rm and}\quad \EE\left(\beta_{n+1}(\widehat{p}_n)^{2k}~|~\widehat{p}_n=x\right)\leq \frac{\iota_2(k)}{N}
$$
for some finite constants $\iota_1(k),\iota_2(k)$. Also note that due to the binomial formula, we have
$$
\begin{array}{l}
\displaystyle
\EE\left(\left(\frac{\phi(p_n)}{p_{n+1}}\right)^k~|~p_n,\widehat{p}_n\right)\\
\\
\displaystyle=1+\EE\left(\left(
\widehat{\alpha}_n(p_n)+\beta_{n+1}(\widehat{p}_n)\right)~|~p_n,\widehat{p}_n\right)+
\sum_{2\leq l\leq k}\left(\begin{array}{c}
k\\
l
\end{array}\right)~\EE\left(\left(
\widehat{\alpha}_n(p_n)+\beta_{n+1}(\widehat{p}_n)\right)^{l}~|~p_n,\widehat{p}_n\right)
\end{array}.$$
We may estimate the above summands, for any $l\geq 2$ by
$$
\begin{array}{l}
\displaystyle
\EE\left(\left(
\widehat{\alpha}_n(p_n)+\beta_{n+1}(\widehat{p}_n)\right)^{l}~|~p_n,\widehat{p}_n\right)\\
\\
\displaystyle\leq \vert \widehat{\alpha}_n(p_n)\vert^{l}+\frac{1}{2}~\sum_{0\leq i< l}\left(\begin{array}{c}
l\\
i
\end{array}\right)\left(
\widehat{\alpha}_n(p_n)^{2l}+\EE\left( \left(\beta_{n+1}(\widehat{p}_n)\right)^{2(l-i)}~|~p_n,\widehat{p}_n\right)\right)\\
\\
\displaystyle\leq  \vert \widehat{\alpha}_n(p_n)\vert^{l}+\frac{1}{2}~\sum_{0\leq i< l}\left(\begin{array}{c}.
l\\
i
\end{array}\right)\left(
\widehat{\alpha}_n(p_n)^{2l}+\frac{\iota_{2}(l-i)}{N}\right)
\end{array}$$
The end of the proof of (\ref{ratio-k}) is now a direct consequence of the moment estimates stated above.
\cqfd

\section{Ensemble Kalman state estimates}

\subsection{A Feynman-Kac formula}\label{change-probab-sec}

\begin{theo}\label{ref-theo-FK}
Let $X_n$ be a Markov chain on some measurable state space $E$ with Markov transitions $M(x,dy)$ and starting at some state $X_0=x$. 
Also let $H(x)$ be some positive measurable function on $E$. 
Assume there exists some function $h$ and some parameters $\epsilon_h\in ]0,1[$ and $\kappa_h \in ]0, \infty[$ such that
\begin{equation}\label{HhYP}
H^h(x):=H(x)~h(x)M(1/h)(x)\leq (1-\epsilon_h)\quad \mbox{and}\quad M(1/h)(x)\leq \kappa_h ~M(1/h)(y).
\end{equation}
Then, for any $n\geq 1$ and any bounded measurable function $F$ on the path space $E^n$ we have the Feynman-Kac change of measure formula
$$
\begin{array}{l}
\displaystyle
\EE\left(F(X_1,\ldots,X_n)~\left(\prod_{1\leq k\leq n}H(X_k)\right)~|~X_0=x\right)\\
\\
\displaystyle=\EE\left(F(X^{1/h}_1,\ldots,X^{1/h}_n)~~\frac{M(h^{-1})(X^{1/h}_{0})}{M(h^{-1})(X^{1/h}_{n})}\left(
\prod_{1\leq k\leq n}H^h(X^{1/h}_{k})\right)~|~X^{1/h}_0=x\right).
\end{array}
$$
In the above display, $X^{1/h}_n$ stands for the Markov chain with Markov transitions
\begin{equation}\label{Markov-h}
M^{1/h}(x,dy)=M(x,dy)~\frac{h^{-1}(y)}{M(h^{-1})(x)}.
 \end{equation}
In addition, we have the uniform exponential decay estimate
\begin{equation}\label{HhYP-d}
\sup_{x} \EE\left(\prod_{1\leq k\leq n}H(X_k)~|~X_0=x\right)\leq \kappa_h~(1-\epsilon_h)^n.
 \end{equation}
\end{theo}
\proof
We set $h^{-1}(x)=1/h(x)$. Let $X^{1/h}_n$ be  the Markov chain starting at $X^{1/h}_0=X_0=x_0$ whose Markov transitions $M^{1/h}$ are given by (\ref{Markov-h}) so that
\begin{equation}
 M^{1/h}(h)(x)=\frac{1}{M(h^{-1})(x)}.
 \label{def-Mh}
\end{equation}
Observe that
$$
\begin{array}{l}
\PP^{1/h}_n(d(x_0,\ldots,x_n))\\
\\
:=\PP((X^{1/h}_0,\ldots,X^{1/h}_n)\in d(x_0,\ldots,x_n))\\
\\
\displaystyle=\eta_0(dx_0)~M(x_0,dx_1)~\frac{h^{-1}(x_1)}{M(h^{-1})(x_0)}\ldots M(x_{n-1},dx_n)~\frac{h^{-1}(x_n)}{M(h^{-1})(x_{n-1})}.
\end{array}$$
Rewritten in terms of expectations we have
$$
\EE\left(F(X_0,\ldots,X_n)~\prod_{1\leq k\leq n}\frac{h^{-1}(X_k)}{M(h^{-1})(X_{k-1})}\right)=
\EE\left(F(X_0^{1/h},\ldots,X_n^{1/h})\right).
$$
Observe that
$$
\begin{array}{l}
\PP^{1/h}_n(d(x_0,\ldots,x_n))\\
\\
\displaystyle=\PP_n(d(x_0,\ldots,x_n))~\left(\frac{h^{-1}(x_1)}{M(h^{-1})(x_0)}\ldots ~\frac{h^{-1}(x_n)}{M(h^{-1})(x_{n-1})} \right)
\end{array}$$
with
$$
\PP_n(d(x_0,\ldots,x_n)):=\eta_0(dx_0)~M(x_0,dx_1)~\ldots M(x_{n-1},dx_n).
$$
This clearly implies that
$$
\begin{array}{l}
\PP_n(d(x_0,\ldots,x_n))\\
\\
\displaystyle=\PP^{1/h}_n(d(x_0,\ldots,x_n))~\left(\left(h(x_1)~M(h^{-1})(x_0)\right)\ldots ~\left(h(x_n)~M(h^{-1})(x_{n-1})\right) \right)\\
\\
\displaystyle=\PP^{1/h}_n(d(x_0,\ldots,x_n))~\left(\frac{h(x_1)}{M^{1/h}(h)(x_0)}\ldots ~\frac{h(x_n)}{M^{1/h}(h)(x_{n-1})} \right),
\end{array}$$
where we have used \eqref{def-Mh} in the final step.
Rewritten in terms of expectations, the above formula takes the following form
$$
\EE\left(F(X_0^{1/h},\ldots,X_n^{1/h})\prod_{1\leq k\leq n}\frac{h(X_k^{1/h})}{M^{1/h}(h)(X^{1/h}_{k-1})}\right)=\EE\left(F(X_0,\ldots,X_n)\right).
$$
Replacing $F(X_0,\ldots,X_n)$ by
$$
F(X_0,\ldots,X_n)\prod_{1\leq k\leq n}H(X_k),
$$
and again recalling \eqref{def-Mh}, we obtain the formula
$$
\begin{array}{l}
\displaystyle
\EE\left(F(X_0,\ldots,X_n)\prod_{1\leq k\leq n}H(X_k)\right)\\
\\
\displaystyle=\EE\left(F(X_0^{1/h},\ldots,X^{1/h}_n)\left(\prod_{1\leq k\leq n}H(X^{1/h}_k)\right)\left(\prod_{1\leq k\leq n}\frac{h(X_k^{1/h})}{M^{1/h}(h)(X^{1/h}_{k-1})}\right)\right)\\\\
\displaystyle=\EE\left(F(X^{1/h}_0,\ldots,X^{1/h}_n)\left(\prod_{1\leq k\leq n}(H h)(X^{1/h}_k)\right)\left(\prod_{0\leq k< n}M(h^{-1})(X^{1/h}_{k})\right)\right)\\
\\
\displaystyle=\EE\left(F(X^{1/h}_0,\ldots,X^{1/h}_n)~M(h^{-1})(X^{1/h}_{0})(Hh)(X_{n}^{1/h})\left(\prod_{1\leq k< n}(Hh)(X^{1/h}_k)M(h^{-1})(X^{1/h}_{k})\right)\right)\\
\\
\displaystyle=\EE\left(F(X^{1/h}_0,\ldots,X^{1/h}_n)~\frac{M(h^{-1})(X^{1/h}_{0})}{M(h^{-1})(X^{1/h}_n)}\left(\prod_{1\leq k\leq n}(Hh)(X^{1/h}_k)M(h^{-1})(X^{1/h}_{k})\right)\right).
\end{array}
$$
This ends the proof of the  the Feynman-Kac change of measure formula.
The exponential decay estimate (\ref{HhYP-d}) is now a direct consequence of the regularity condition (\ref{HhYP}).
\cqfd

We are now in position to prove the uniform almost sure exponential decays stated in Theorem~\ref{key-theo}.

{\bf Proof of first inequality in \eqref{expo-decay-p}}
Up to a time change, it suffices to prove the result for $l=1$ and for any given initial condition $p_0=x$.
Recall the Markov transition, $\mathcal{P}$, of the stochastic Riccati process $p_n$.
We choose in Theorem~\ref{ref-theo-FK} the function
$$
H(x)=\left(\frac{\vert A\vert }{1+Sx}\right)^k\quad \mbox{\rm and}\quad
h(x)=x^k.
$$
Suppose that $\vert A\vert<1$.  In this situation, we have
$$
\EE\left(\prod_{1\leq l\leq n}\left(\frac{\vert A\vert}{1+Sp_l}\right)^k\right)\leq (1-\epsilon_k)^n\quad \mbox{\rm with}\quad \epsilon_k:=1-\vert A\vert^k.
$$
Now assume that $\vert A\vert \geq 1$.  In this situation, we have
$
\vert A\vert<(A^2+RS)$.
Observe that
$$
H^h(x)=H(x)~h(x)\mathcal{P}(h^{-1})(x)=\left(\frac{\vert A\vert }{1+Sx}\right)^k~\EE\left(\left(\frac{p_0}{p_1}\right)^k~|~p_0=x\right).
$$
Using \eqref{ratio-k} we check that
$$
H^h(x)=\left(\frac{\vert A\vert }{1+Sx}~\frac{x}{\phi(x)}\right)^k~\EE\left(\left(\frac{\phi(p_0)}{p_1}\right)^k~|~p_0=x\right)\leq \left(\frac{\vert A\vert }{1+Sx}~\frac{x}{\phi(x)}\right)^k~\left(1+\frac{\iota_k}{N}\right).
$$
This yields the estimate
$$
H^h(x)\leq \left(\frac{\vert A\vert }{(A^2+RS)+R/x}\right)^k~\left(1+\frac{\iota_k}{N}\right)\leq \left(\frac{\vert A\vert }{A^2+RS}\right)^k~\left(1+\frac{\iota_k}{N}\right).
$$
Choosing $N$ such that
\begin{equation*}
\displaystyle N\geq N_{k}:=\iota_k \left(1-\left(\frac{\vert A\vert}{A^2+RS}\right)^k\right)^{-1}
\end{equation*}
yields
\begin{equation*}
\left(1+\frac{\iota_k}{N}\right)\leq 1+\sqrt{\epsilon_k},
\quad\mbox{\rm
with}
\quad
\sqrt{\epsilon_k}:=\left(1-\left(\frac{\vert A\vert}{A^2+RS}\right)^k\right).
\end{equation*}
We readily check that
\begin{eqnarray*}
H^h(x)&\leq& \left(1-\sqrt{\epsilon_k}\right)\left(1+\sqrt{\epsilon_k}\right)=1-\epsilon_k<1.
\end{eqnarray*}
This implies that
$$
\EE\left(\left(\prod_{1\leq l\leq n}\frac{\vert A\vert}{1+Sp_l}\right)^k~|~p_0=x\right)\leq \left(1-\epsilon_k\right)^{n-1}\EE\left(\frac{\Pa(h^{-1})(x)}{\Pa(h^{-1})(X^{1/h}_{n-1})}\right).
$$
On the other hand, using (\ref{ricc-lower-cv}) and (\ref{ratio-k}), for any even parameter $N>2(k+2)$ we have
$$
\frac{1}{(A^2/S+R)^k}
\leq 
\Pa(h^{-1})(x)=\frac{1}{\phi(p_0)^k}~\EE\left(\left(\frac{\phi(p_0)}{p_1}\right)^k~|~p_0=x\right)
\leq \frac{1}{R^k}~\left(1+\frac{\iota_k}{N}\right)
$$
We conclude that
$$
\EE\left(\left(\prod_{1\leq l\leq n}\frac{\vert A\vert}{1+Sp_l}\right)^k\right)\leq   \left(1-\epsilon_k\right)^{n-1}~\left(A^2R/S+R^2\right)^k \left(1+\frac{\iota_k}{N}\right).
$$
This ends the proof of the exponential decays estimates  stated in the l.h.s. of (\ref{expo-decay-p}).
\cqfd

\subsection{Fluctuation analysis}\label{theo-enkf-wX-proof}

\begin{lem}
For any $k\geq 1$ there exists  some parameter $N_k\geq 1$ such that for any $N\geq N_k$ and $n\geq 0$ we have the time uniform estimates
\begin{equation}\label{estimates-m-X-x-k}
\EE\left(\vert m_{n}-X_{n}\vert^k\right)\vee\EE\left(\vert \widehat{m}_{n}-X_{n}\vert^k\right)\vee \EE\left(\vert Y_n-Cm_{n}\vert^k\right)\leq \iota_k.
\end{equation}
\end{lem}
\proof
Using (\ref{ref-key-obs}) we check the recursion
$$
(m_{n+1}-X_{n+1})=\alpha_n~(X_n-m_n)+\beta_{n+1},
$$
with the random variables
\begin{align*}
\alpha_n&:=A(1-g_nC)=\frac{A}{1+Sp_n} \\
\beta_{n+1}&:=Ag_nDV_n-BW_{n+1}+\frac{1}{\sqrt{N+1}}~\left(A\widehat{\upsilon}_n+\upsilon_{n+1}\right).
\end{align*}
This yields the formula
\begin{equation}\label{expo-decay-sum-ref}
(m_{n+1}-X_{n+1})=\Ea_{0,n}~(X_0-m_0)+\beta_{n+1}+\sum_{0\leq k\leq n}\Ea_{k,n}~\beta_{k}.
\end{equation}
 The raw moment estimates of the differences
 $(m_{n}-X_{n})$ and $(Y_n-Cm_{n})$ stated in \eqref{estimates-m-X-x-k} are now a direct consequence of the exponential decay estimate stated in \eqref{expo-decay-p}. The raw moment estimates of the difference
 $(\widehat{m}_{n}-X_{n})$ is easily checked using  the following decomposition
 $$
  \widehat{m}_{n}-X_{n}=\frac{A}{1+Sp_n}~(m_n-X_n)+g_nCDV_n+\frac{1}{\sqrt{N+1}}~\widehat{\upsilon}_n
 $$
 This ends the proof of the lemma.
 \cqfd
 
 Note that this completes the proof of Theorem \ref{key-theo}. Now we come to the proof of the Lyapunov estimate (\ref{Lyap-form}).
 
 \bigskip
 
 {\bf Proof of (\ref{Lyap-form}):}
  Combining (\ref{expo-decay-sum-ref}) with the uniform almost sure exponential decays stated in (\ref{expo-decay-p}) we check the uniform estimate
 $$
\EE\left(\vert M_{n+1}\vert~|~(M_0,p_0)\right)\leq \EE(\Ea_{0,n}~|~p_0)~\vert M_0\vert+\iota_1\leq \iota_1+\iota_2 (1-\epsilon_2)^{n}\vert M_0\vert.$$
Thus, for any $\epsilon\in ]0,1]$ there exists some  time horizon $n_{\epsilon}\geq 1$ such that
for any $n\geq n_{\epsilon}$ we have
 $$
\EE\left(\vert M_{n}\vert~|~(M_0,p_0)\right)\leq \epsilon~\vert M_0\vert +\iota.
$$
The estimate (\ref{Lyap-form}) now comes from the fact that the one-step Riccati map $\phi$ is uniformly bounded.
 \cqfd

{\bf Proof of Theorem~\ref{theo-enkf-wX}:}
First note that
\begin{equation}\label{deco-XX-0}
m_{n+1}-\widehat{X}_{n+1}^-=\displaystyle A~(\widehat{m}_{n}-\widehat{X}_n)+\frac{1}{\sqrt{N+1}}~\upsilon_{n+1}.
\end{equation}
We also have the decomposition
\begin{equation}\label{deco-XX}
\begin{array}{l}
\displaystyle
(\widehat{m}_{n}-\widehat{X}_{n})\\
\\
=\displaystyle (m_n-\widehat{X}^-_{n})+(g_n-G_n)~(Y_n-Cm_n)-G_n~C(m_n-\widehat{X}^-_{n})+\frac{1}{\sqrt{N+1}}~\widehat{\upsilon}_n\\
\\
=\displaystyle (1-G_nC)(m_n-\widehat{X}^-_{n})+\frac{1}{\sqrt{N}}~\sqrt{N}(g_n-G_n)~(Y_n-Cm_n)+\frac{1}{\sqrt{N+1}}~\widehat{\upsilon}_n.
\end{array}
\end{equation}
This yields the recursion
$$
(\widehat{m}_{n}-\widehat{X}_{n}) =\alpha_n~(\widehat{m}_{n-1}-\widehat{X}_{n-1})+\frac{1}{\sqrt{N}}~ \beta_n,
$$
with the parameters
\begin{eqnarray*}
\alpha_n&:=&A(1-G_nC)= \frac{A}{1+SP_n}\\
\beta_n&:=&\frac{1}{\sqrt{1+1/N}}~ \left((1-G_nC) \upsilon_{n}+\widehat{\upsilon}_n\right)
+\sqrt{N}(g_n-G_n)~(Y_n-Cm_n).
\end{eqnarray*}
We conclude that
$$
(\widehat{m}_{n}-\widehat{X}_{n}) =E_n(P_0)~(\widehat{m}_{0}-\widehat{X}_{0})+\frac{1}{\sqrt{N}}~\sum_{1\leq l\leq n}~ E_{l,n}(P_0)~\beta_l.
$$
Finally, observe that
$$
(\widehat{m}_{0}-\widehat{X}_{0})
=\displaystyle (1-G_0C)(m_0-\widehat{X}^-_{0})+\frac{1}{\sqrt{N}}~\left(\sqrt{N}(g_0-G_0)~(Y_0-Cm_0)+\frac{1}{\sqrt{1+1/N}}~\widehat{\upsilon}_0\right).
$$
The estimate (\ref{estimates-m-wX-k}) is now a consequence of the exponential decays (\ref{est-expo-Ea}), the uniform bounds (\ref{estimates-m-X-x-k})
and the mean error estimates stated in Theorem~\ref{theo-u-p}.

\medskip

The proof of \eqref{estimates-bias-m-wX-k} follows similar arguments. Combining (\ref{deco-XX-0}) and (\ref{deco-XX}) we check the decomposition
$$
\begin{array}{l}
\displaystyle
m_{n+1}-\widehat{X}_{n+1}^-=A~ (1-G_nC)(m_n-\widehat{X}^-_{n})+A(g_n-G_n)~(Y_n-C\widehat{X}_n^-)\\
\\
\hskip2cm\displaystyle -\frac{1}{N}~\sqrt{N}~A(g_n-G_n)C~~\sqrt{N}(m_n-\widehat{X}^-_n)+\frac{1}{\sqrt{N+1}}~A\widehat{\upsilon}_n+\frac{1}{\sqrt{N+1}}~\upsilon_{n+1}.
\end{array}
$$
Taking the expectations we obtain the formula
$$
\begin{array}{l}
\displaystyle
(m_{n+1}^o-\widehat{X}_{n+1}^-)=\alpha_n~(m_n^o-\widehat{X}^-_{n})+\beta_n/N
\end{array}
$$
with the parameters $(\alpha_n,\beta_n)$ now given by
\begin{eqnarray*}
\alpha_n&:=&A(1-G_nC)= {A}/{(1+SP_n)}\\
\beta_n&:=&A~N(\EE(g_n)-G_n)~(Y_n-C\widehat{X}_n^-)-AC~\EE\left(\sqrt{N}~(g_n-G_n)\sqrt{N}(m_n-\widehat{X}^-_n)~|~\Ya_{n-1}\right).
\end{eqnarray*}

Using the uniform estimates stated in Theorem~\ref{theo-u-p}
and the gain bias estimates (\ref{bias-estimates}), we obtain, for $k \ge 1$
   $$
 \sup_{n\geq 0}\EE\left(\vert\beta_n\vert^k\right)^{1/k}\leq \iota_k~(1\vee P_0)^2.
  $$
The end of the proof of (\ref{estimates-bias-m-wX-k}) now follows the same lines of arguments as the proof of the estimates (\ref{estimates-m-wX-k}), thus we leave the details as an exercise for the reader.
\cqfd

{\bf Proof of Theorem~\ref{term-theo}:}
First note that from the proof of \eqref{estimates-m-wX-k}, we have the following decomposition
\begin{equation}
\sqrt{N}(\widehat m_n - \widehat{X}_n) = \mathcal E_n(P_0)\sqrt{N}(\widehat m_0 - \widehat{X}_0) + \sum_{1 \le \ell \le n}\mathcal E_{\ell, n}(P_0)\beta_\ell,
\label{decomp1}
\end{equation}
where we recall that 
\begin{align*}
\mathcal E_{\ell, n}(p) &= \prod_{\ell < k \le n}\frac{A}{1 + SP_\ell(p)},\\
\beta_\ell &= \frac{1}{\sqrt{1+1/N}}~ \left((1-G_nC) \upsilon_{n}+\widehat{\upsilon}_n\right)
+\sqrt{N}(g_n-G_n)~(Y_n-Cm_n).
\end{align*}
We first deal with $\sqrt{N}(\widehat m_0 - \widehat{X}_0)$ on the right hand side of \eqref{decomp1}. Again, from the proof of \eqref{estimates-m-wX-k}, recall that 
\begin{equation}
\sqrt{N}(\widehat{m}_{0}-\widehat{X}_{0})
=\displaystyle (1-G_0C)\sqrt{N}(m_0-\widehat{X}^-_{0})+\frac{1}{\sqrt{1+1/N}}~\widehat{\upsilon}_0+\sqrt{N}(g_0-G_0)~(Y_0-Cm_0).
\label{decomp2}
\end{equation}
Further observe that
$$
(Y_0-Cm_0)=(Y_0-C\widehat{X}^-_0)+C(\widehat{X}^-_0-m_0)\quad \mbox{\rm and}\quad
(\widehat{X}^-_0-m_0)
\longrightarrow_{N\rightarrow\infty} 0\quad \mbox{\rm a.e.}.$$ 
Applying Slutsky's lemma, this yields the weak convergence
$$
\sqrt{N}(g_0-G_0)~C(\widehat{X}^-_0-m_0)\hooklongrightarrow_{N\rightarrow\infty} 0.
$$
Using the fluctuation theorems \ref{theo-loc-tcl} and \ref{theo-clt-p-intro} we also have
\begin{align*}
\sqrt{N}(m_0-\widehat{X}^-_{0}) &\hooklongrightarrow_{N\rightarrow\infty} \UU_0,\\
\sqrt{N}(g_0-G_0)(Y_0-C\widehat{X}^-_0) &\hooklongrightarrow_{N\rightarrow\infty} \GG_0(Y_0-C\widehat{X}^-_0),\\
\frac{1}{\sqrt{1+1/N}}~\widehat{\upsilon}_0 &\hooklongrightarrow_{N\rightarrow\infty} \widehat\UU_0.
\end{align*}
In the same vein, applying the continuous mapping theorem or the $\delta$-method (see for instance Theorem 9.3.3 and Lemma 9.3.1 in~\cite{fk}) we check that
$$
\sqrt{N}(\widehat{m}_{0}-\widehat{X}_{0})\hooklongrightarrow_{N\rightarrow\infty}  (1-G_0C) \UU_0+\widehat\UU_0+
\GG_0(Y_0-C\widehat{X}^-_0).
$$

To deal with the sum on the right hand side of \eqref{decomp1}, again, thanks to Theorem \ref{theo-loc-tcl},  and arguing as above $\beta_\ell$ converges weakly to 
$$
(1-G_{\ell}C)\UU_{\ell} + \widehat\UU_{\ell}+\GG_{\ell} (Y_{\ell}-C\widehat{X}_{\ell}^-).
$$
In the same vein,   we check that 
\begin{align*}
\widehat{\XX}_n^N &\hooklongrightarrow_{N\rightarrow\infty}  \widehat{\XX}_n
:= \sum_{0 \le \ell \le n}E_{\ell, n}(P_0)\left((1-G_{\ell}C)\UU_{\ell} + \widehat\UU_{\ell}+\GG_{\ell} (Y_{\ell}-C\widehat{X}_{\ell}^-)\right).
\end{align*}
The proof of the fluctuation theorem at the level of the process
in the sense of convergence of finite dimensional distributions is conducted combining the continuous mapping theorem with Slutsky's lemma.
The proof of the recursive formulation (\ref{XX-w-XX-theo})  is also easily checked using the decompositions (\ref{deco-XX-0}) and (\ref{deco-XX}).  This ends the proof of the theorem.\cqfd

\newpage
\section*{Appendix}
\subsection*{Proof of Theorem~\ref{stab-riccati-stoch}}

For the identity function $\Ua(p):=p$, due to \eqref{ul-bounds}, we have
$$
\Pa(\Ua)(p)=
\EE(\Ua(p_{n+1})~|~p_n=p)=\phi(p)\leq \alpha=A^2/S+R.
$$
 In addition, for any $\epsilon\in ]0,1]$ we have
$$
\Ua_{\epsilon}:=1+\frac{\epsilon}{\alpha}~\Ua\geq 1\Longrightarrow \Pa(\Ua_{\epsilon})(p)\leq 1+\epsilon\Longrightarrow \Pa(\Ua_{\epsilon})\leq 1+\epsilon~ \Ua_{\epsilon}.
$$
The above estimate implies that $\Ua_{\epsilon}\geq 1$ is a Lyapunov function.
On the other hand, $\Pa(p,dq)=k(p,q)~dq$ has a continuous density $k(p,q)>0$ with respect to the variables $(p,q)$ and the Lebesgue measure $dq$ on $\RR_+$. Thus, for any compact set $\varpi\subset \RR_+$ there exists some function $\varphi$ from $[0,\infty[$ into itself such that 
$$
k_{\varpi}(q):=\inf_{p\in \varpi}k(p,q)=k(\varphi_{\varpi}(q),q)>0 \Longrightarrow 0<\epsilon_{\varpi}:=\int~k_{\varpi}(q)~dq<1.
$$
This implies that
$$
\Pa(p,dq) \ge \epsilon_{\varpi}~\nu_{\varpi}(dq)\quad \mbox{\rm with}\quad \nu_{\varpi}(dq):=\frac{k_{\varpi}(q)~dq}{\int~k_{\varpi}(p)~dp},
$$
from which we prove the Dobrushin local contraction property 
$$
\varpi_r:=[0,r]\Longrightarrow
\beta_{r}(\Pa):=\sup_{(p,q)\in \varpi_r^2}\Vert \Ka(p, \point)-\Ka(q, \point)  \Vert_{\tiny tv}<1-\epsilon_{\varpi}<1.
$$
For any $\epsilon_1,\epsilon_2\in ]0,1]$ we set
$$
\Ua_{\epsilon_1,\epsilon_2}:=\epsilon_1\Ua_{\epsilon_2}\Longrightarrow\Pa(\Ua_{\epsilon_1,\epsilon_2})\leq \epsilon_1\left(1+\epsilon_2\Ua_{\epsilon_2}\right)\leq \epsilon_1+\epsilon_2~\Ua_{\epsilon_1,\epsilon_2}.
$$
Choosing the parameters $(r,\epsilon_2)$ such that $(1/r+\epsilon_2)<1$ for any $\epsilon_1\in [0,1[$ we can check that
$$
p,q\geq r\Longrightarrow\frac{\Vert \Pa(p, \cdot) - \Pa(q, \cdot) \Vert_{\Ua_{\epsilon_1,\epsilon_2}}}{1 + \Ua_{\epsilon_1,\epsilon_2}(p) + \Ua_{\epsilon_1,\epsilon_2}(q)}<1.
$$
The detailed proof of the above assertion can be found in section 8.2.7 in~\cite{dm-penev}.
In addition, for any $x,y\leq r$ we also have
$$
\epsilon_1<\frac{1-\beta_r(\Ka)}{4r}\Longrightarrow \frac{\Vert \Pa(x, \cdot) - \Pa(y, \cdot) \Vert_{\Ua_{\epsilon_1,\epsilon_2}}}{1 + \Ua_{\epsilon_1,\epsilon_2}(x) + \Ua_{\epsilon_1,\epsilon_2}(y)}\leq \beta_r(\Ka)+4\epsilon_1 r<1.
$$
We conclude that 
$$
(1/r+\epsilon_2)<1\quad\mbox{\rm and}\quad \epsilon_1<\frac{1-\beta_r(\Pa)}{4r}\quad\Longrightarrow\quad
\beta_{\Ua_{\epsilon_1,\epsilon_2}}(\Pa)<1,
$$
which concludes the proof. \cqfd

\subsection*{Proof of Lemma~\ref{lem-ricc}}
Following~\cite{brand}, we have
\begin{equation}\label{f1-pw}
\phi^n(x)+v=\lambda_2~\frac{\left(x+(v-\lambda_1)\right)~\left(1-\lambda^{n+1}\right)+(\lambda_2-\lambda_1)~\lambda^{n+1}}{\left(x+(v-\lambda_1)\right)~\left(1-\lambda^{n}\right)+(\lambda_2-\lambda_1)~\lambda^{n}}.
\end{equation}
Next, we check that the r.h.s. expression in the above display is non negative. Notice that
$$
\phi^n(x)=\frac{\left(x+(v-\lambda_1)\right)\left(\lambda_2~\left(1-\lambda^{n+1}\right)-v~\left(1-\lambda^{n}\right)\right)-(v-\lambda_1)(\lambda_2-\lambda_1)~\lambda^{n}}{\left(x+(v-\lambda_1)\right)~\left(1-\lambda^{n}\right)+(\lambda_2-\lambda_1)~\lambda^{n}}.
$$
Rewritten in a slightly different form (\ref{f1-pw}) takes the following form
$$
\phi^n(x)=(\lambda_2-v)+
\frac{\left(x+(v-\lambda_1)\right)~-(\lambda_2-\lambda_1)}{\left(x+(v-\lambda_1)\right)~\left(1-\lambda^{n}\right)+(\lambda_2-\lambda_1)~\lambda^{n}}~\left(\lambda_2-\lambda_1\right)~\lambda^n.
$$
This yields the formula
$$
\phi^n(x)-r=
\frac{x-r}{\left(x+(v-\lambda_1)\right)~\left(1-\lambda^{n}\right)+(\lambda_2-\lambda_1)~\lambda^{n}}~\left(\lambda_2-\lambda_1\right)~\lambda^n,
$$
which ends the proof of \eqref{sol-ricc}. 
Recalling that $\phi$ is an increasing function, we check that
$$
\phi(0)=\frac{b}{d}\leq 
\phi(x)\leq \lim_{x\rightarrow\infty}\phi(x)=\frac{a}{c}\quad \mbox{\rm and}\quad r=\phi(r)\leq \frac{a}{c},
$$
which ends the proof of \eqref{ul-bounds}.

For any  $x,y\in \RR_+$ we also have
$$
\begin{array}{l}
\displaystyle\frac{\phi^n(x)-\phi^n(y)}{x-y}\\
\\
\displaystyle=\frac{(\lambda_2-\lambda_1)^2~\lambda^n}{\left[\left(x+(v-\lambda_1)\right)~\left(1-\lambda^{n}\right)+(\lambda_2-\lambda_1)~\lambda^{n}\right]\left[\left(y+(v-\lambda_1)\right)~\left(1-\lambda^{n}\right)+(\lambda_2-\lambda_1)~\lambda^{n}\right]}.
\end{array}$$

This ends the proof of the r.h.s. estimate in \eqref{expo-bounds}. The proof of the l.h.s. estimate in (\ref{expo-bounds}) is a direct consequence of the easily checked estimate
$$
\vert \phi^n(x)-r\vert \leq \vert x-r\vert ~
\frac{\left(\lambda_2-\lambda_1\right)}{\left(x+(v-\lambda_1)\right)~\wedge (\lambda_2-\lambda_1)}~~\lambda^n.
$$

We check (\ref{2nd-order}) using the second order formula
$$
\begin{array}{l}
\displaystyle\frac{\phi^n(x)-\phi^n(y)}{x-y}\\
\\
\displaystyle=\frac{(\lambda_2-\lambda_1)^2~\lambda^n}{\left[\left(y+(v-\lambda_1)\right)~\left(1-\lambda^{n}\right)+(\lambda_2-\lambda_1)~\lambda^{n}\right]^2}\\
\\
\displaystyle-\frac{(\lambda_2-\lambda_1)^2~\lambda^n}{\left[\left(y+(v-\lambda_1)\right)~\left(1-\lambda^{n}\right)+(\lambda_2-\lambda_1)~\lambda^{n}\right]^2}\times\left(\frac{(x-y)\left(1-\lambda^{n}\right)}{\left[\left(x+(v-\lambda_1)\right)~\left(1-\lambda^{n}\right)+(\lambda_2-\lambda_1)~\lambda^{n}\right]}\right).
\end{array}$$
Finally observe that
$$
(\ref{def-rho})\Longrightarrow
\partial \phi^n(x)=\partial (\phi(\phi^{n-1}(x)))=(\partial \phi)(\phi^{n-1}(x))~\partial \phi^{n-1}(x)=\prod_{0\leq k<n}\frac{\rho}{(c\phi^{k}(x)+d)^2}
$$
with the convention $\phi^0(x)=x$.
This ends the proof of (\ref{1st-order-exp}) and the lemma.\cqfd

\subsection*{Proof of Lemma~\ref{lem-Helmert}}
It clearly  suffices to prove (\ref{charact-Z}) for $\omega_1=0$.
 For any $u,v\in \RR$ we have
$$
\EE\left(\exp{\left(i(uZ_1+vZ^2_1)\right)}\right)=(1-2iv)^{-1/2}~\exp{\left(-\frac{u^2}{2(1-2iv)}\right)}.
$$
Then, for any $u,v\in\RR$ we have $$
\EE\left(\exp{\left(
\frac{u}{N}\sum_{1\leq i\leq N}Z_i+
\frac{v}{N}\sum_{1\leq i\leq N}Z^2_i\right)}\right)=\left(\frac{1}{\sqrt{1-2v/N}}\right)^N
~\exp{\left(-\frac{u^2}{2N(1-2v/N)}\right)}.
$$
This yields the characteristic function formula
$$
\EE\left(\exp{\left(i\left(u~\widecheck{\Za}+v~\widetilde{\Za}\right)\right)}\right)~\exp{\left(\frac{u^2+v^2}{2}\right)}=\exp{\left(\epsilon_1(u,v)+\epsilon_2(v)\right)}.
$$
with the functions
\begin{eqnarray*}
\epsilon_1(u,v)&:=&\frac{u^2}{2}\left(1-\frac{1}{\left(1-2i{v}/{\sqrt{2N}}\right)}\right)=-i~\frac{u^2v}{\sqrt{2N}}~\frac{1}{\left(1-2i{v}/{\sqrt{2N}}\right)}\\
&&\\
\epsilon_1(v)&:=&\frac{v^2}{2}-i\frac{Nv}{\sqrt{2N}}-\frac{N}{2}\log{\left(1-\frac{2iv}{\sqrt{2N}}\right)}.
\end{eqnarray*}
Using the estimates
$$
\vert-\log{(1-ix)}-\left(ix-\frac{x^2}{2}\right)\vert\leq 2~\vert x \vert^3\quad\mbox{\rm and}\quad \vert ix/(1-iy)\vert=\frac{\vert x\vert}{\sqrt{1+y^2}}\leq \vert x\vert
$$
we also check that
$$
\vert\epsilon_1(u,v)\vert\leq \frac{u^2\vert v\vert}{\sqrt{2N}}\quad\mbox{\rm and}\quad
\vert\epsilon_2(v)\vert\leq \frac{N}{2}~\left(2\frac{v}{\sqrt{2N}}\right)^3=\vert v\vert^3\sqrt{\frac{2}{N}}.
$$
We conclude that
$$
\EE\left(e^{i\left(u\,\widecheck{\Za}+v\,\widetilde{\Za}\right)}\right)-e^{-\frac{u^2+v^2}{2}}=\left(e^{\epsilon_1(u,v)+\epsilon_2(v)}-1\right)~e^{-\frac{u^2+v^2}{2}}.
$$
Finally, using the estimate
$$
\vert e^z-1\vert\leq \vert z\vert~\int_0^1\vert e^{tz}\vert~dt\leq \vert z\vert~\int_0^1\vert e^{t~\mbox{\footnotesize Re}(z)}\vert~ dt \leq
\vert z\vert~e^{\vert z\vert},
$$
we check that
$$
\left\vert \EE\left(e^{i\left(u\,\widecheck{\Za}+v\,\widetilde{\Za}\right)}\right)-e^{-\frac{u^2+v^2}{2}}\right\vert\leq \epsilon(u,v)~e^{-\frac{u^2+v^2}{2}}.
$$
This ends the proof of the lemma.\cqfd

\end{document}